\documentclass{article}

\parskip \smallskipamount

\usepackage{hyperref}
\hypersetup{
    colorlinks=true,
    linktocpage=true,
    breaklinks=true,
    urlcolor= blue,
    linkcolor= red,
    citecolor=blue,
}
\usepackage{amsmath,amsthm,amssymb}

\usepackage{tikz}
\usepackage{enumerate}
\usetikzlibrary{shapes}
\usetikzlibrary{decorations.pathreplacing}
\usetikzlibrary{patterns}

\newtheorem{lemma}{Lemma}
\newtheorem{theorem}{Theorem}
\newtheorem{proposition}{Proposition}
\newtheorem{corollary}{Corollary}
\newcommand{\N}{\mathbb{N}}
\newcommand{\Z}{\mathbb{Z}}
\newcommand{\E}{\mathbb{E}}
\newcommand{\PP}{\mathbb{P}}

\newcommand{\s}{\mathfrak{s}}

\newcommand{\drawsocle}[1]
{
\draw[very thick] (-.3,.3) -- node[midway,below, align=center]{#1} (9.3,.3);
}

\newcommand{\draweta}[1]{
\drawsocle{#1}
\foreach \y in {(1,1),(1,2),(2,1),(6,1),(6,2),(6,3),(8,1)}
{
\filldraw \y circle(.4);
}
}
\newcommand{\drawcopyeta}{\draweta{copy of~$\eta$}}

\newcommand{\drawbigpart}
{
\draw[ultra thick] (0,0) -- node[midway,below]{1} (4,0);
\filldraw (2,2) circle(1.5);
}
\newcommand{\drawtwobigpart}
{
\draw[ultra thick] (0,0) -- node[midway,below]{2} (4,0);
\filldraw (2,2) circle(1.5);
\filldraw (2,5.5) circle(1.5);
}
\newcommand{\drawbigvoid}
{
\draw[very thick,dotted,blue!50!white, line cap=round] (0,0) --
node[midway,below]{0} (4,0);
}

\newcommand{\drawbigsleep}
{
\draw[ultra thick] (0,0) -- node[midway,below]{$\s$} (4,0);
\draw (2,2) node[scale=2] {$\s$};
}

\newcommand{\comment}[1]{}

\title{Macroscopic flow out of a segment for Activated Random Walks in
dimension~1}

\author{Nicolas Forien\thanks{
  CEREMADE, CNRS, Université Paris-Dauphine, Université PSL,
75016 Paris, France}}

\begin{document}

\maketitle

\begin{abstract}
Activated Random Walk is a system of interacting particles
which presents a phase transition and a conjectured phenomenon of
self-organized criticality.
In this note, we prove that, in dimension 1,
in the supercritical case, when a segment is
stabilized with particles
being killed when they jump out of the segment, a positive fraction of the
particles leaves the segment with positive probability.

This was already known to be a sufficient condition for being in the active
phase of the model, and the result of this paper is that this
condition is also necessary, except maybe precisely at the critical point.
This result can also be seen as a partial answer to some of the many
conjectures which connect the different points of view on the phase
transition of the model.
\end{abstract}

\par\noindent \emph{Keywords and phrases.} Activated random walks, phase transition, self-organized criticality.
\par\noindent MSC 2020 \emph{subject classifications.} 60K35, 82B26.

\section{Introduction}

We begin with a brief informal presentation of some aspects of the
model.
The reader who is familiar with Activated Random Walks may skip
the two following subsections, passing directly to
Subsection~\ref{section_results} where our results are presented.
Some illustrated sketches of proofs are given in
Subsection~\ref{section_sketches}.

\subsection{Presentation of the model}

The model of Activated Random Walks consists of particles performing
independent random walks on a graph, which fall asleep with a certain
rate and get reactivated in the presence of other particles on the
same site.
The model was popularized by Rolla, Sidoravicius and
Dickman~\cite{Rolla08, RS12, DRS10}, and can be seen as a variant of
the frog model~\cite{AMP02, AMP02b}.
Its study is motivated by its connection with the concept of
self-organized criticality, which was introduced by the
physicists Bak, Tang and Wiesenfeld~\cite{BTW87} to describe physical
systems which present
a critical-like behaviour but without the need to tune the parameters
of the system to particular values (like is the case for an ordinary
phase transition).
To illustrate this concept, Bak, Tang and Wiesenfeld introduced an
interacting particle system called the Abelian sandpile model, which
is a close cousin of Activated Random Walks.
Yet, one of the key differences between these two models is their mixing behaviour.
The mixing properties of Activated Random Walks are investigated
in~\cite{LL21,BS22}, and it turns out that this model mixes faster
than the Abelian sandpile.
This can explain why Activated Random Walks are expected to have a
behaviour which is more universal, in that it is
less sensitive to microscopic details of the system.

Let us now define informally the Activated Random Walk model
on~$\Z^d$.
A configuration of the model consists of a given number of
particles
on each site of~$\Z^d$, each of these particles being in one of
two possible
states: active or sleeping.

The model evolves as follows.
Each active particle performs a
continuous-time random walk on~$\Z^d$ with jump rate~$1$, with a
certain jump distribution~$P:\Z^d\times\Z^d\to[0,1]$ (such
that~$\sum_{y\in\Z^d}P(x,\,y)=1$ for each~$x\in\Z^d$).
This means that, after a random time distributed as an exponential
with parameter~$1$, the active particle at~$x$ jumps to some other
site, the probability of jumping from~$x$ to~$y$ being~$P(x,\,y)$.
This distribution is called translation-invariant if~$P(x,\,y)$ is a
function of~$y-x$ only.

In parallel, each active particle also carries another 
exponential
clock with a certain parameter~$\lambda>0$ and, when this clock rings,
if there are no other particles on the same site, the particle falls
asleep (otherwise, if the particle is not alone, nothing happens).
A sleeping particle stops moving (its continuous-time random walk is
somewhat paused), until it wakes up, which happens when another
particle arrives on the same site. Then, the reactivated particle
resumes its
continuous-time random walk with jump rate~$1$.
Equivalently, one may also consider that a particle can fall
asleep even when it is not alone on a site but, whenever this happens, the
particle is instantaneously waken up by the presence of the other
particles.

Thus, there can never be two sleeping particles at a same site.
Hence, at every time~$t\geq 0$, the configuration of the model at
time~$t$ can by encoded into a
function~$\eta_t:\Z^d\to\N\cup\{\s\}$, where~$\eta_t(x)=k\in\N$ means
that there are~$k$ active particles at the site~$x$,
while~$\eta_t(x)=\s$ means that there is one sleeping particle at~$x$.
Note that, with this notation, particles are indistinguishable: we
only keep track of the number and states of particles on each site,
but not of the individual trajectory of each particle.

Regarding the initial configuration~$\eta_0$, various setups are
interesting to consider.
One possibility is to take~$\eta_0$ which
follows a translation-invariant and ergodic probability distribution
on the set
of all possible configurations, with a finite mean number of
particles per site.
This initial configuration may have only active particles, or both
active and sleeping particles.
Another case of interest is that of~$\eta_0$ with only finitely many
particles, for
example~$n$ particles on the origin.
The model can also be defined on different graphs, or with slight
modifications of the dynamics, like for example adding a sink
vertex where particles get trapped forever.

For a rigorous construction of the process~$(\eta_t)_{t\geq 0}$, we refer the reader to~\cite{RS12}, or to
the review~\cite{Rolla20}.
See also~\cite{LS24} for a presentation of various different settings
of interest and fascinating conjectures connecting
these different points of view on the model.

\subsection{Phase transition}

Let us consider the model on~$\Z^d$ starting with~$\eta_0$ following a
translation-ergodic distribution with mean particle density~$\zeta$.
Depending on the sleep rate~$\lambda$, on the jump distribution~$P$
and on the particle density~$\zeta$, the model can exhibit very different
behaviours.
A natural question is: if we start with only active particles, do they
eventually all fall asleep, or is activity maintained forever?
Note that, on the infinite lattice~$\Z^d$, almost surely there exists
no finite time when all the particles are sleeping.
However, we have the following notion of fixation: we say that the
system fixates if the origin is visited finitely many times by an
active particle during the evolution of the process.
Then, up to events of~$0$ probability, fixation is
equivalent to the configuration on every finite set of~$\Z^d$
eventually being constant and with only sleeping particles or, if we
follow the trajectory of each individual particle, fixation also turns out
to be equivalent to each particle walking only a finite number of
steps, or to one given particle walking only a finite number of steps~\cite{AGG10}.
If the system does not fixate, we say that the system stays active.

Due to the ergodicity assumption on~$\eta_0$, the
probability of fixation can only be~$0$ or~$1$ (see~\cite{RS12}).
Thus, we can have two different regimes, depending on the sleep
rate~$\lambda$, on the jump distribution~$P$ and on the law of the initial
configuration: either the system almost surely fixates (this
regime is called the fixating phase, or stable phase), or the system almost
surely stays active (this is called the active phase, or exploding
phase).
Moreover, we have the following key result about this phase
transition:

\begin{theorem}[\cite{RSZ19}]
\label{thm_phase_transition}
In any dimension~$d\geq 1$, for every sleep
rate~$\lambda\in(0,\infty]$ and every translation-invariant jump
distribution~$P$ which
generates all~$\Z^d$, there exists~$\zeta_c$ such that, for every
translation-ergodic initial distribution with no sleeping particles
and an average density of active particles~$\zeta$, the Activated
Random Walk model on~$\Z^d$ with sleep rate~$\lambda$ almost surely
fixates if~$\zeta<\zeta_c$, whereas it almost surely stays active
if~$\zeta>\zeta_c$.
\end{theorem}

This result shows in particular that the critical
density is in some sense universal, in that it depends on the
initial configuration only through the mean density of
particles~$\zeta$.
Thus, to study this critical density, it is enough to consider the
particular case where the configuration is i.i.d., with a given
probability distribution on~$\N$ with finite mean.

An important challenge in the study of this phase transition is to
relate the property of fixation, which concerns the model on
the infinite lattice with infinitely many particles, to some
finite counterparts of the model.
A key example is the sufficient condition for activity given by
Theorem~\ref{thm_condition_Mn} below.

On the finite box~$V_n=(-n/2,\,n/2\,]^d\cap\Z^d$,
let us consider a
variant of the model where particles are
killed and removed from the system when they jump out of~$V_n$ (or
equivalently, we can consider the model on~$\Z^d$ where particles are
frozen outside~$V_n$, so
that particles which start out of~$V_n$ or which jump out of~$V_n$
are frozen forever and cannot move any more).
Let~$M_n$ count the number of particles that jump out of~$V_n$ when we
let this system evolve until all the sites of~$V_n$ become stable (a site~$x$ is called stable if it is either empty or it contains a
sleeping particle).

\begin{theorem}[\cite{RT18}]
\label{thm_condition_Mn}
With the notation defined above, for every
sleep rate~$\lambda$ and every translation-invariant jump
distribution~$P$ which
generates all~$\Z^d$, if the initial configuration~$\eta_0$ is i.i.d.\ 
and if
\begin{equation}
\label{condition_Mn}
\limsup_{n\to\infty}\,
\frac{\E\big[M_n\big]}{|V_n|}
\ >\
0 
\,,
\end{equation}
then the model on~$\Z^d$ with sleep rate~$\lambda$, jump
distribution~$P$ and initial configuration~$\eta_0$
almost surely stays active.
\end{theorem}

This result relies on the following intuitive idea: if with positive
probability a large box looses a
positive fraction of its particles during stabilization, then a
particle starting at the
origin in the model on~$\Z^d$ has a positive probability of walking
arbitrarily far away, which shows that the system stays active with
positive probability, and thus with probability~$1$ (because we have a
0-1 law).

\subsection{Main results}
\label{section_results}

The main result of this paper consists in the addition of a reciprocal
to the implication of Theorem~\ref{thm_condition_Mn}, in the
particular case of dimension~$1$.
Recall that~$M_n$ denotes the number of particles that jump out
of~$V_n=(-n/2,\,n/2\,]\cap\Z$ during the stabilization of~$V_n$ with
particles being killed upon leaving~$V_n$.

\begin{theorem}
\label{main_thm}
In dimension~$d=1$, for every sleep rate~$\lambda>0$ and every
nearest-neighbour translation-invariant jump distribution~$P$,
if the initial configuration~$\eta_0$ is i.i.d.\ with
mean~$\zeta$ and all particles are initially active,
then we have the equivalence:
$$
\zeta
\ >\ \zeta_c
\quad\Longleftrightarrow\quad
\liminf_{n\to\infty}\,
\frac{\E[M_n]}{|V_n|}
\ >\ 
0
\,.
$$
\end{theorem}

This shows that the sufficient condition for activity given by
Theorem~\ref{thm_condition_Mn} is also necessary, except maybe
exactly at the critical point.
Indeed, very few things are known rigorously about the critical
regime~$\zeta=\zeta_c$, with the exception of the particular case of
directed walks in dimension~$1$ starting with~$\eta_0$ i.i.d., for which a proof of non-fixation at
criticality,
due to Hoffman and Sidoravicius, appears in~\cite{CRS14}.

Theorem~\ref{main_thm} answers a conjecture of Levine and Silvestri~\cite{LS24} (in the
particular case of dimension~$1$), showing
that the density~$\zeta_w$ that they define in Section 5.3, and which
corresponds to the infimum of the~$\zeta$ for which
condition~\eqref{condition_Mn} holds when~$\eta_0$ is i.i.d.\ Poisson, is in fact equal to the critical
density~$\zeta_c$.

Our result is made more precise by the following theorem, which indicates an explicit positive fraction which exits with
positive probability, as a function of the sleep rate~$\lambda$ and
the density~$\zeta$.
For every deterministic initial configuration~$\eta:V_n\to\N$ (with
only active particles), let us denote by~\smash{$\|\eta\|=\sum_{x\in
V_n}\eta(x)$} the total number of particles in the
configuration~$\eta$, and let us write~$\PP_\eta$ for the probability
relative to the system started with deterministic initial
configuration equal to~$\eta$.

\begin{theorem}
\label{thm_explicit}
In dimension~$d=1$, for every sleep rate~$\lambda>0$ and every
nearest-neighbour translation-invariant jump distribution~$P$, for
every~$\zeta>\zeta_c$ we have
$$
\forall\,\varepsilon\in\bigg[\,0,\,\frac{\lambda(\zeta-\zeta_c)}{4(1+\lambda)\zeta_c}\bigg)
\qquad
\liminf_{n\to\infty}\ 
\inf_{\substack{\eta:V_n\to\N\,:\\ \|\eta\|\geq\zeta n}}\ 
\PP_{\eta}\big(M_n> \varepsilon n\big)
\ \geq\ 
1-\frac{\zeta_c}{\zeta}
\bigg(1+\frac{4(1+\lambda)\varepsilon}{\lambda}\bigg)
\ >\ 
0
\,.
$$
\end{theorem}

To show this, as a first step we prove the following result, which gives
an explicit upper bound on the probability that no particle exits
during stabilization.
This bound is not optimal but, as explained later, it allows
us to obtain the bound of Theorem~\ref{thm_explicit}.

\begin{theorem}
\label{thm_no_exit}
In dimension~$d=1$, for every sleep rate~$\lambda>0$ and every
nearest-neighbour translation-invariant jump distribution~$P$, for
every~$\zeta>\zeta_c$, for every~$n\geq 1$ and every
initial configuration~$\eta:V_n\to\N$ with~$\|\eta\|\geq\zeta n$
particles, initially all active, we have
$$
\PP_\eta\big(M_n=0\big)
\ \leq\ 
\frac{\zeta_c}{\zeta}
\,.
$$
\end{theorem}

Let us stress that the bounds of Theorem~\ref{thm_explicit}
and~\ref{thm_no_exit} hold for any deterministic initial configuration
with at least~$\zeta n$ particles, which includes in particular the
case of interest where all the particles start from the
origin (see comments about this in Section~\ref{section_perspectives}).

To show that Theorem~\ref{thm_no_exit} implies
Theorem~\ref{thm_explicit}, we use the following
important observation: adding empty intervals around~$V_n$ where
particles are not allowed to sleep and
stabilizing the configuration in this enlarged segment does not
increase, in distribution, the number of particles which jump out.
This is the content of Lemma~\ref{lemma_NML}, that we postpone to
a later Section but which can be of independent interest.

Lastly, in the course of the proof of Theorem~\ref{main_thm}, to deal
with the
case~$\zeta=\zeta_c$, we establish the following fact, which can also be
of independent interest:

\begin{proposition}
\label{fact_critical}
In any dimension~$d\geq 1$, for every sleep rate~$\lambda>0$ and
every translation-invariant jump distribution~$P$ on~$\Z^d$ whose support generates all the
group~$\Z^d$, if~$\eta_0$ is i.i.d.\  with
mean~$\zeta_c$ then we have
$$
\lim_{n\to\infty}\,
\frac{\E[M_n]}{|V_n|}
\ =\ 
0
\,.
$$
\end{proposition}

Notably, this last result shows that, at least in the case of directed
walks in dimension~$1$, for which it is known that there is no
fixation at criticality (see the remark following the statement of
Theorem~\ref{main_thm}), the sufficient
condition for activity given by Theorem~\ref{thm_condition_Mn} is not
necessary: in this case, at~$\zeta=\zeta_c$, the system is active despite the condition~\eqref{condition_Mn} not being satisfied.
That is to say, the system is active if and only
if~$\zeta\geq\zeta_c$,
whereas condition~\eqref{condition_Mn} is equivalent
to the strict inequality~$\zeta>\zeta_c$.

\subsection{Some perspectives}
\label{section_perspectives}

Since the seminal works which established general properties of
the phase transition, various techniques have been
developed to study Activated Random Walks.
In particular, a series of
works~\cite{RS12,ST18,ARS19,Taggi19,BGH18,HRR20,FG22,Hu22,AFG22}
established that the critical density
is always strictly between~$0$ and~$1$ and obtained bounds
on~$\zeta_c$ as a function of~$\lambda$.
But many of the techniques used only work far from criticality, when the
density
is either much larger or much smaller than~$\zeta_c$, and
few results have been proved to hold up to the critical
density.
For example,~\cite{BGHR19} shows that the model on the torus
stabilizes fast when~$\zeta$ is very small, and slowly when~$\zeta$ is
close to~$1$, but we lack sharper results about a transition exactly
at~$\zeta_c$ from fast to slow stabilization.

Some exceptions giving insight about the behaviour at or close
to~$\zeta_c$ are the study of the critical regime in the case of
directed walks in one dimension (see~\cite{CRS14}, with an argument
due to Hoffman and Sidoravicius, and~\cite{CR20}), the continuity
of~$\zeta_c$ as a function of~$\lambda$~\cite{Taggi23}, and the recent
work~\cite{JMT23} which considers the model on the complete graph and
computes the exact value of the critical density.

In this regard, the results of the present article have the merit to
hold up to the critical point.
However, the bounds presented here are far from being optimal, and there
remains a lot of space for improvement.
For example, Theorem~\ref{main_thm} can be seen as a partial answer to
the so-called hockey stick conjecture (conjecture~17\ in~\cite{LS24}),
which predicts that~$M_n/|V_n|$ should converge in probability
to~$\max(0,\,\zeta-\zeta_c)$, at least in the particular case when the
initial distribution is i.i.d.\ Poisson.

Similarly, the bound given in Theorem~\ref{thm_no_exit} is not optimal, and it is expected that when~$\zeta>\zeta_c$, the probability
that~$M_n=0$ in fact decays exponentially fast with~$n$ (see
conjecture~20 of~\cite{LS24}).

Note that Theorem~\ref{thm_no_exit} can also be seen as a partial
answer to the so-called ball conjecture (see conjectures~1 and~12
in~\cite{LS24}).
This conjecture predicts that, when starting
with~$n$ particles at the
origin, if we let these particles stabilize in~$\Z^d$, the random set
of visited sites~$A_n$ is such that, for every~$\varepsilon>0$, with
probability tending to~$1$ as~$n\to\infty$, the set~$A_n$ contains all
the sites of~$\Z^d$ that belong to the origin-centred Euclidean ball
of volume~$(1-\varepsilon)n/\zeta_c$ and is contained in the
origin-centred Euclidean ball of volume~$(1+\varepsilon)n/\zeta_c$.
Theorem~\ref{thm_no_exit} implies that the probability that~$A_n$ is
included into the ball of volume~$(1-\varepsilon)n/\zeta_c$ is less
than~$1-\varepsilon$.
Note that another partial result was obtained in this direction
in~\cite{LS21}, also
in dimension~$1$, showing an inner and an outer bound on~$A_n$.

Last but not least, the proofs of the present paper are very specific
to the one-dimensional case, and it would be interesting to obtain at
least similar results in higher dimension.
See more comments on this in Section~\ref{section-high-dim}.

\textbf{Update:} While this paper was under review, a new result
appeared~\cite{HJJ1,HJJ2} which proves the hockey stick conjecture and
the ball conjecture for activated random walk in dimension 1, showing
that the various definitions of the critical density coincide, and
strengthening the bounds of Theorems~\ref{thm_explicit}
and~\ref{thm_no_exit}.
Their technique involves a detailed study of the odometers and a
comparison with a percolation process, and it is very different from
the ideas of the present paper.
We thus hope that the ideas that we present here are nevertheless of
interest.

\section{Sketches of the proofs}
\label{section_sketches}

Let us now summarize the strategy of the proofs.
We start with Theorem~\ref{thm_no_exit}, before explaining how
the other results follow.

We will make extensive use of the abelian property of the model, which
allows us to choose the order with which the particles act: as
long as there is at least one active particle, we pick one with a
certain rule, and with probability~$\lambda/(1+\lambda)$ it falls
asleep if it is alone, whereas with probability~$1/(1+\lambda)$ it makes a
jump distributed according to the prescribed jump distribution~$P$.
The abelian property states that, if the randomness of these sleeps
and jumps is quenched into an array with a list of instructions above
each site, then once this array is fixed, the final stable
configuration and the number of instructions used at each site do not
depend on the
order with which the moves are performed.
The formalism of this quenched array of instructions is presented in
Section~\ref{section_tools}, where the abelian property corresponds to
Lemma~\ref{lemma_abelian}.

We also make use of the monotonicity property of the model with
respect to enforced activation.
This means that, stabilizing the
configuration by forcing sometimes sleeping particles to wake up, we
obtain
upper bounds on the number of instructions needed to
stabilize (see Lemma~\ref{lemma_monotonicity}).

\subsection{Sketch of the proof of Theorem~\ref{thm_no_exit}}
\label{subsec_sketch_no_exit}

Let~$\lambda>0$, let~$P$ be a nearest-neighbour translation-invariant jump distribution,
let~$\zeta>\zeta_c$ and~$n\geq 1$ and consider a fixed
deterministic initial configuration~$\eta:V_n\to\N$
with~$\|\eta\|=\zeta n$.
Let us write~$p=\PP_\eta(M_n=0)$.
This means that, when we stabilize the configuration~$\eta$ in~$V_n$
with particles ignored once they jump out of~$V_n$, with
probability~$p$ all the particles fall asleep inside~$V_n$ with no
particle jumping out: we call this a good stabilization.
The stabilization of~$\eta$ is called bad if at least one particle jumps out
of~$V_n$ (see first part of Figure~\ref{fig_eta}).
Our aim is to show that~$p\leq\zeta_c/\zeta$.

The strategy is the following: for every~$\zeta'<p\zeta$, we construct
a random initial configuration on~$\Z$ with particle density~$\zeta'$
and we show that this configuration fixates, which implies that~$\zeta'\leq\zeta_c$.
This being true for every~$\zeta'<p\zeta$, we deduce
that~$p\zeta\leq\zeta_c$.

\begin{figure}
\begin{center}
\begin{tikzpicture}[scale=.2]
\filldraw[fill=gray!15,thick] (-32,10) rectangle (41,23);
\begin{scope}[shift={(-9,5)}]

\begin{scope}[shift={(-20,11.5)}]
\draweta{$\eta:V_n\to\N$ with\\
$\frac{\|\eta\|}n=\zeta>\zeta_c$}
\end{scope}

\begin{scope}[green!50!black]
\draw[ultra thick,->,line cap=round] (-7,13) --
node[midway,above]{with prob.~$p$} (9,14.5);
\begin{scope}[shift={(13,14)}]
\drawsocle{}
\foreach \y in {(0,1),(1,1),(3,1),(5,1),(6,1),(8,1),(9,1)}
{
\draw \y node{$\s$};
}
\end{scope}
\draw (24,15) node[right,align=left]{\textbf{Good stabilization:}\\
no particle jumps out of~$V_n$.};
\end{scope}

\begin{scope}[red!50!black]
\draw[ultra thick,->,line cap=round] (-7,10) --
node[midway,below]{with prob.~$1-p$} (9,8.5);
\begin{scope}[shift={(13,7)}]
\drawsocle{}
\foreach \y in {(1,1),(3,1),(6,1),(8,1)}
{
\draw \y node{$\s$};
}
\filldraw (-1,1) circle(.4) (-1,2) circle(.4) (10,1) circle(.4);
\draw[very thick,->] (2,2.5) to[bend right] (-.7,2.6);
\draw[very thick,->] (8,2.5) to[bend left] (9.7,1.6);
\end{scope}
\draw (24,8) node[right,align=left]{\textbf{Bad stabilization:}\\
some particles jump out of~$V_n$.};
\end{scope}

\end{scope}

\filldraw[thick,fill=gray!15] (-35,-6.5) rectangle (44,8);

\foreach \x in {-3,-1,1,2}
{
\begin{scope}[shift={(10*\x,0)}]
\drawcopyeta
\end{scope}
}

\foreach \x in {-2,0,3}
{
\begin{scope}[shift={(10*\x,0)}]
\draw[very thick,dotted,blue!50!white, line cap=round] (-.2,.3)
--node[midway,below]{empty block} (9.2,.3);
\end{scope}
}
\draw[blue!50!white] (4.5,.3) node{$\bullet$} node[above]{origin $0$};
\draw (-32,.3) node{\textbf{...}} (41,.3) node{\textbf{...}};

\draw (4.5,6.5) node{Random block configuration on~$\Z$ with
particle density~$q\zeta$:};
\draw (4.5,-5) node{Each block starts with a copy
of~$\eta$ with probability~$q$.};
\draw (4.5,-9) node{$\longrightarrow$ If~$q<p$, then with positive probability, the origin is never visited, whence~$q\zeta\leq\zeta_c$.};

\end{tikzpicture}
\end{center}
\caption{
\label{fig_eta}
Outline of the proof of Theorem~\ref{thm_no_exit}: assuming that a
given deterministic initial configuration~\smash{$\eta:V_n\to\N$} produces a
good stabilization with probability~$p$, we build a block
configuration on~$\Z$ with density~$q\zeta$ which fixates if~$q<p$, showing
that~$q\zeta\leq \zeta_c$.
Since this holds for every~$q<p$, we
get~$p\zeta\leq\zeta_c$.
}
\end{figure}

The configuration that we consider is a ``block configuration''
represented in the second part of Figure~\ref{fig_eta} and constructed
as follows.
Let~$q=\zeta'/\zeta$.
For each block of the form~$V_n+kn$ for~$k\in\Z$, with probability~$q$
this block starts with a copy of the configuration~$\eta$, and with
probability~$1-q$ the block starts empty, independently for different
blocks.

Then, our strategy is to stabilize this configuration one block after
another, always choosing in priority an unstable block which is
as close as possible to the origin.
This unstable block, which contains a copy
of~$\eta$, gives a good stabilization with probability~$p$.
If this happens, we turn to the next block.
If a bad stabilization occurs, that is to say, if some particles
jump out of their block, then we force all the particles of this block
to walk, forbidding them to fall asleep, until we obtain again a copy
of~$\eta$ translated on the following block in the direction of the origin.
If this block in the direction of the origin was already occupied by
sleeping particles resulting from a previous good stabilization, then
we force these sleeping particles to wake up and walk until they form
a copy of~$\eta$ on the following block towards the origin, leaving
behind them a copy of~$\eta$.
And if again we arrive on a block already
occupied by sleeping particles, we repeat this displacement operation
until we go back to a situation where, in each block, we have either a
copy of~$\eta$ (active block) or only sleeping particles (stable
block).

We want to show that, with positive probability, the blocks can be
stabilized with this strategy with no particle ever visiting the
block of the origin.
To study this stabilization block by block, we couple the model with
another instance of activated
random walk, that we call ``coarse-grained model'', which has an initial
configuration i.i.d.\ Bernoulli with parameter~$q$ (each site
corresponds to one block in the original model), a sleep
rate~$\lambda'$ such that~$\lambda'/(1+\lambda')=p$ (a particle
falling asleep corresponds to the good stabilization of a block in the
original
model) and a jump distribution such that all sites~$x\neq 0$ jump
towards the origin with probability~$1$ (because at each bad
stabilization we force the particles to move towards the origin).

This coarse-graining coupling is illustrated in
Figure~\ref{fig_coupling} and formalized in
Proposition~\ref{prop_coupling}, which is postponed to
Section~\ref{section_no_exit} because it requires some
notation introduced in Section~\ref{section_tools}.
More precisely, this proposition shows that the original model and
the coarse-grained model can be coupled in such a way that, if the
origin is never
visited in the coarse-grained model, then the sites in the block of
the origin are never visited in the original model.

\begin{figure}
\begin{center}
\begin{tikzpicture}[scale=.147]


\foreach \x in {-3,-2,2}
{
\begin{scope}[shift={(10*\x,0)}]
\drawcopyeta
\end{scope}
\begin{scope}[shift={(60+5*\x,0)}]
\drawbigpart
\end{scope}
}

\foreach \x in {-1,0,1,3}
{
\begin{scope}[shift={(10*\x,0)}]
\draw[very thick,dotted,blue!50!white, line cap=round] (-.2,.3) --
node[midway,below]{empty} (9.2,.3);
\end{scope}
\begin{scope}[shift={(60+5*\x,0)}]
\drawbigvoid
\end{scope}
}

\draw[blue!50!white] (4.5,.3) node{$\bullet$} node[above]{origin $0$};

\begin{scope}[shift={(0,-10)}]

\foreach \x in {-3,2}
{
\begin{scope}[shift={(10*\x,0)}]
\drawcopyeta
\end{scope}
\begin{scope}[shift={(60+5*\x,0)}]
\drawbigpart
\end{scope}
}

\begin{scope}[green!50!black]
\foreach \x in {-2}
{
\begin{scope}[shift={(10*\x,0)}]
\drawsocle{stable}
\foreach \y in {0,1,3,4,6,7,8}
{
\draw (\y,1) node[scale=.8]{$\s$};
}
\end{scope}
\begin{scope}[shift={(60+5*\x,0)}]
\drawbigsleep
\end{scope}
}
\end{scope}

\foreach \x in {-1,0,1,3}
{
\begin{scope}[shift={(10*\x,0)}]
\draw[very thick,dotted,blue!50!white, line cap=round] (-.2,.3) --
node[midway,below]{empty} (9.2,.3);
\end{scope}
\begin{scope}[shift={(60+5*\x,0)}]
\drawbigvoid
\end{scope}
}

\end{scope}

\begin{scope}[shift={(0,-20)}]

\foreach \x in {-3}
{
\begin{scope}[shift={(10*\x,0)}]
\drawcopyeta
\end{scope}
\begin{scope}[shift={(60+5*\x,0)}]
\end{scope}
}

\begin{scope}[green!50!black]
\foreach \x in {-2}
{
\begin{scope}[shift={(10*\x,0)}]
\drawsocle{stable}
\foreach \y in {0,1,3,4,6,7,8}
{
\draw (\y,1) node[scale=.8]{$\s$};
}
\end{scope}
\begin{scope}[shift={(60+5*\x,0)}]
\end{scope}
}
\end{scope}

\begin{scope}[red!50!black]
\foreach \x in {2}
{
\begin{scope}[shift={(10*\x,0)}]
\drawsocle{bad stab.}
\foreach \y in {(-1,1),(10,2),(10,1)}
{
\filldraw \y circle(.4);
}
\foreach \y in {2,3,5,8}
{
\draw (\y,1) node[scale=.8]{$\s$};
}
\draw[ultra thick,->] (6,2.5) to[bend left] (9.7,2.6);
\draw[ultra thick,->] (3,2.5) to[bend right] (-.7,1.6);
\end{scope}
}
\end{scope}

\foreach \x in {-1,0}
{
\begin{scope}[shift={(10*\x,0)}]
\draw[very thick,dotted,blue!50!white, line cap=round] (-.2,.3) --
node[midway,below]{empty} (9.2,.3);
\end{scope}
\begin{scope}[shift={(60+5*\x,0)}]
\end{scope}
}
\foreach \x in {1,3}
{
\begin{scope}[shift={(10*\x,0)}]
\draw[very thick,dotted,blue!50!white, line cap=round] (-.2,.3) -- (9.2,.3);
\end{scope}
\begin{scope}[shift={(60+5*\x,0)}]
\end{scope}
}

\end{scope}

\begin{scope}[shift={(0,-30)}]

\foreach \x in {-3,1}
{
\begin{scope}[shift={(10*\x,0)}]
\drawcopyeta
\end{scope}
\begin{scope}[shift={(60+5*\x,0)}]
\drawbigpart
\end{scope}
}

\begin{scope}[green!50!black]
\foreach \x in {-2}
{
\begin{scope}[shift={(10*\x,0)}]
\drawsocle{stable}
\foreach \y in {0,1,3,4,6,7,8}
{
\draw (\y,1) node[scale=.8]{$\s$};
}
\end{scope}
\begin{scope}[shift={(60+5*\x,0)}]
\drawbigsleep
\end{scope}
}
\end{scope}

\foreach \x in {-1,0,2,3}
{
\begin{scope}[shift={(10*\x,0)}]
\draw[very thick,dotted,blue!50!white, line cap=round] (-.2,.3) --
node[midway,below]{empty} (9.2,.3);
\end{scope}
\begin{scope}[shift={(60+5*\x,0)}]
\drawbigvoid
\end{scope}
}

\end{scope}

\begin{scope}[shift={(0,-40)}]

\foreach \x in {-3}
{
\begin{scope}[shift={(10*\x,0)}]
\drawcopyeta
\end{scope}
\begin{scope}[shift={(60+5*\x,0)}]
\drawbigpart
\end{scope}
}

\begin{scope}[green!50!black]
\foreach \x in {-2}
{
\begin{scope}[shift={(10*\x,0)}]
\drawsocle{stable}
\foreach \y in {0,1,3,4,6,7,8}
{
\draw (\y,1) node[scale=.8]{$\s$};
}
\end{scope}
\begin{scope}[shift={(60+5*\x,0)}]
\drawbigsleep
\end{scope}
}
\end{scope}

\begin{scope}[green!50!black]
\foreach \x in {1}
{
\begin{scope}[shift={(10*\x,0)}]
\drawsocle{stable}
\foreach \y in {1,2,3,4,6,8,9}
{
\draw (\y,1) node[scale=.8]{$\s$};
}
\end{scope}
\begin{scope}[shift={(60+5*\x,0)}]
\drawbigsleep
\end{scope}
}
\end{scope}

\foreach \x in {-1,0,2,3}
{
\begin{scope}[shift={(10*\x,0)}]
\draw[very thick,dotted,blue!50!white, line cap=round] (-.2,.3) --
node[midway,below]{empty} (9.2,.3);
\end{scope}
\begin{scope}[shift={(60+5*\x,0)}]
\drawbigvoid
\end{scope}
}

\end{scope}

\begin{scope}[shift={(0,-50)}]

\foreach \x in {}
{
\begin{scope}[shift={(10*\x,0)}]
\drawcopyeta
\end{scope}
\begin{scope}[shift={(60+5*\x,0)}]
\end{scope}
}

\begin{scope}[red!50!black]
\foreach \x in {-3}
{
\begin{scope}[shift={(10*\x,0)}]
\drawsocle{bad stab.}
\foreach \y in {(-1,1),(10,2)}
{
\filldraw \y circle(.4);
}
\foreach \y in {1,4,5,6,7}
{
\draw (\y,1) node[scale=.8]{$\s$};
}
\draw[ultra thick,->] (6,2.5) to[bend left] (9.7,2.5); 
\draw[ultra thick,->] (3,2.5) to[bend right] (-.7,1.6);
\end{scope}
}
\end{scope}

\begin{scope}[green!50!black]
\foreach \x in {-2}
{
\begin{scope}[shift={(10*\x,0)}]
\drawsocle{disturbed}
\foreach \y in {1,3,4,6,7,8}
{
\draw (\y,1) node[scale=.8]{$\s$};
}
\filldraw (0,1) circle(.4);
\end{scope}
\begin{scope}[shift={(60+5*\x,0)}]
\end{scope}
}
\end{scope}

\begin{scope}[green!50!black]
\foreach \x in {1}
{
\begin{scope}[shift={(10*\x,0)}]
\drawsocle{stable}
\foreach \y in {1,2,3,4,6,8,9}
{
\draw (\y,1) node[scale=.8]{$\s$};
}
\end{scope}
\begin{scope}[shift={(60+5*\x,0)}]
\end{scope}
}
\end{scope}

\foreach \x in {-1,0,2,3}
{
\begin{scope}[shift={(10*\x,0)}]
\draw[very thick,dotted,blue!50!white, line cap=round] (-.2,.3) --
node[midway,below]{empty} (9.2,.3);
\end{scope}
\begin{scope}[shift={(60+5*\x,0)}]
\end{scope}
}

\end{scope}

\begin{scope}[shift={(0,-60)}]

\foreach \x in {-2}
{
\begin{scope}[shift={(10*\x,0)}]
\draweta{$2\|\eta\|$ part.}
\foreach \y in {(0,1),(1,3),(3,1),(4,1),(6,4),(7,1),(8,2)}
{
\filldraw \y circle(.4);
}

\end{scope}
\begin{scope}[shift={(60+5*\x,0)}]
\drawtwobigpart
\end{scope}
}

\begin{scope}[green!50!black]
\foreach \x in {1}
{
\begin{scope}[shift={(10*\x,0)}]
\drawsocle{stable}
\foreach \y in {1,2,3,4,6,8,9}
{
\draw (\y,1) node[scale=.8]{$\s$};
}
\end{scope}
\begin{scope}[shift={(60+5*\x,0)}]
\drawbigsleep
\end{scope}
}
\end{scope}

\foreach \x in {-3,-1,0,2,3}
{
\begin{scope}[shift={(10*\x,0)}]
\draw[very thick,dotted,blue!50!white, line cap=round] (-.2,.3) --
node[midway,below]{empty} (9.2,.3);
\end{scope}
\begin{scope}[shift={(60+5*\x,0)}]
\drawbigvoid
\end{scope}
}

\end{scope}

\begin{scope}[shift={(0,-70)}]

\foreach \x in {-2,-1}
{
\begin{scope}[shift={(10*\x,0)}]
\drawcopyeta
\end{scope}
\begin{scope}[shift={(60+5*\x,0)}]
\drawbigpart
\end{scope}
}

\begin{scope}[green!50!black]
\foreach \x in {1}
{
\begin{scope}[shift={(10*\x,0)}]
\drawsocle{stable}
\foreach \y in {1,2,3,4,6,8,9}
{
\draw (\y,1) node[scale=.8]{$\s$};
}
\end{scope}
\begin{scope}[shift={(60+5*\x,0)}]
\drawbigsleep
\end{scope}
}
\end{scope}

\foreach \x in {-3,0,2,3}
{
\begin{scope}[shift={(10*\x,0)}]
\draw[very thick,dotted,blue!50!white, line cap=round] (-.2,.3) --
node[midway,below]{empty} (9.2,.3);
\end{scope}
\begin{scope}[shift={(60+5*\x,0)}]
\drawbigvoid
\end{scope}
}

\end{scope}

\begin{scope}[shift={(0,-80)}]

\foreach \x in {-2}
{
\begin{scope}[shift={(10*\x,0)}]
\drawcopyeta
\end{scope}
\begin{scope}[shift={(60+5*\x,0)}]
\drawbigpart
\end{scope}
}

\begin{scope}[green!50!black]
\foreach \x in {-1}
{
\begin{scope}[shift={(10*\x,0)}]
\drawsocle{stable}
\foreach \y in {0,1,3,4,6,8,9}
{
\draw (\y,1) node[scale=.8]{$\s$};
}
\end{scope}
\begin{scope}[shift={(60+5*\x,0)}]
\drawbigsleep
\end{scope}
}
\foreach \x in {1}
{
\begin{scope}[shift={(10*\x,0)}]
\drawsocle{stable}
\foreach \y in {1,2,3,4,6,8,9}
{
\draw (\y,1) node[scale=.8]{$\s$};
}
\end{scope}
\begin{scope}[shift={(60+5*\x,0)}]
\drawbigsleep
\end{scope}
}
\end{scope}

\foreach \x in {-3,0,2,3}
{
\begin{scope}[shift={(10*\x,0)}]
\draw[very thick,dotted,blue!50!white, line cap=round] (-.2,.3) --
node[midway,below]{empty} (9.2,.3);
\end{scope}
\begin{scope}[shift={(60+5*\x,0)}]
\drawbigvoid
\end{scope}
}

\end{scope}

\begin{scope}[shift={(0,-90)}]

\begin{scope}[green!50!black]
\foreach \x in {-2}
{
\begin{scope}[shift={(10*\x,0)}]
\drawsocle{stable}
\foreach \y in {1,2,3,4,5,6,9}
{
\draw (\y,1) node[scale=.8]{$\s$};
}
\end{scope}
\begin{scope}[shift={(60+5*\x,0)}]
\drawbigsleep
\end{scope}
}
\foreach \x in {-1}
{
\begin{scope}[shift={(10*\x,0)}]
\drawsocle{stable}
\foreach \y in {0,1,3,4,6,8,9}
{
\draw (\y,1) node[scale=.8]{$\s$};
}
\end{scope}
\begin{scope}[shift={(60+5*\x,0)}]
\drawbigsleep
\end{scope}
}
\foreach \x in {1}
{
\begin{scope}[shift={(10*\x,0)}]
\drawsocle{stable}
\foreach \y in {1,2,3,4,6,8,9}
{
\draw (\y,1) node[scale=.8]{$\s$};
}
\end{scope}
\begin{scope}[shift={(60+5*\x,0)}]
\drawbigsleep
\end{scope}
}
\end{scope}

\foreach \x in {-3,0,2,3}
{
\begin{scope}[shift={(10*\x,0)}]
\draw[very thick,dotted,blue!50!white, line cap=round] (-.2,.3) --
node[midway,below]{empty} (9.2,.3);
\end{scope}
\begin{scope}[shift={(60+5*\x,0)}]
\drawbigvoid
\end{scope}
}

\end{scope}

\draw[green!50!black] (-15.5,-5) node{good stab.\ of~$B_{-2}$} (52,-5)
node{sleep at~$-2$};
\draw[red!50!black] (24.5,-15) node{bad stab.\ of~$B_{2}$} (72,-15)
node{left jump at~$2$} (19.5,-25) node{forced move to~$B_1$};
\draw[green!50!black] (14.5,-35) node{good stab.\ of~$B_1$} (67,-35)
node{sleep at~$1$};
\draw[red!50!black] (-21.5,-45) node{bad stab.\ of~$B_{-3}$} (47,-45)
node{right jump at~$-3$} (-20,-55) node{forced move to~$B_{-2}$};
\draw (-10.5,-65) node{forced move to~$B_{-1}$} (52,-65)
node{(sleep(s) then) right jump at~$-2$};
\draw[green!50!black] (-5.5,-75) node{good stab.\ of~$B_{-1}$} (57,-75)
node{sleep at~$-1$};
\draw[green!50!black] (-15.5,-85) node{good stab.\ of~$B_{-2}$} (52,-85)
node{sleep at~$-2$};

\draw (-25.5,6) node{$B_{-3}$};
\draw (-15.5,6) node{$B_{-2}$};
\draw (-5.5,6) node{$B_{-1}$};
\draw (4.5,6) node{$B_{0}$};
\draw (14.5,6) node{$B_{1}$};
\draw (24.5,6) node{$B_{2}$};
\draw (34.5,6) node{$B_{3}$};
\draw (4.5,9) node{\underline{Block configuration}};

\draw (47,6) node{-3};
\draw (52,6) node{-2};
\draw (57,6) node{-1};
\draw (62,6) node{0};
\draw (67,6) node{1};
\draw (72,6) node{2};
\draw (77,6) node{3};
\draw (62,9) node{\underline{Coarse-grained configuration}};

\end{tikzpicture}
\end{center}
\caption{
\label{fig_coupling}
Coupling between the stabilization of the block configuration, on the
left side, and the stabilization of a coarse-grained configuration, on
the right side (we refer to Section~\ref{subsec_sketch_no_exit} for an
informal sketch and to Proposition~\ref{prop_coupling} for a more
detailed presentation of the coupling).
}
\end{figure}

Once this coupling is established, we use the following result
of Rolla and Sidoravicius:

\begin{theorem}[Theorem~2 in~\cite{RS12}]
\label{lemmaRS}
For every~$\lambda>0$, every nearest-neighbour jump distribution~$P$
and every~$q<\lambda/(1+\lambda)$, in the model with sleep
rate~$\lambda$, jump
distribution~$P$ and initial
configuration i.i.d.\ with mean particle density~$q$, with positive
probability the
origin is never visited.
\end{theorem}

In fact this statement is not exactly Theorem~2 in~\cite{RS12}, but it
directly follows
from the proof (which by the way is stated for a Poisson initial
distribution, but holds for any i.i.d.\ initial distribution).

We apply this result to the coarse-grained model, which has particle
density~$q$ and sleep rate~$\lambda'$ such
that~$\lambda'/(1+\lambda')=p$.
Since~$q<p$, we deduce that with positive probability the origin is never
visited in the coarse-grained model, which implies that with positive
probability the block of the origin is never visited in the original
model, which
implies that the particle density in the block configuration,
namely~$q\zeta=\zeta'$, is at most~$\zeta_c$.

A small issue that we just swept under the carpet is that the
block configuration is not translation-invariant, but this problem is
easily overcome by applying a random translation, as explained in Section~\ref{section_offset}.
Interestingly, our technique is one of the rare applications of the
universality result of Theorem~\ref{thm_phase_transition} to an
initial configuration which is not i.i.d.\ but only
translation-invariant and ergodic.

The proof of Theorem~\ref{thm_no_exit}, with the statement and
proof of Proposition~\ref{prop_coupling} which establishes the
coupling, is the content of Section~\ref{section_no_exit}.

\subsection{Proving that Theorem~\ref{thm_no_exit} implies
Theorem~\ref{thm_explicit}}

To deduce Theorem~\ref{thm_explicit} from Theorem~\ref{thm_no_exit}, the
idea is that, if less than~$\varepsilon n$ particles jump out of the segment,
then, taking a slightly larger segment, with enough empty space around
to easily accommodate these particles which jumped out, we can stabilize
this larger segment with no particles jumping out of it.

We proceed in two steps, as represented in
Figure~\ref{fig_thm_explicit}.
We start with a configuration~$\eta:V_n\to\N$ with
density~$\|\eta\|/n=\zeta>\zeta_c$ and empty strips of
length~$2\alpha n$ on each side of~$V_n$.

\begin{figure}
\begin{center}
\begin{tikzpicture}[scale=0.2,decoration={brace,mirror,amplitude=5}]

\draw[thick, dashed,<->,gray!70!black] (-.5,7) --
node[midway,above]{$n$} (19.5,7);
\foreach \x in {-5,-10,20,25}
{
\begin{scope}[shift={(\x,0)}]
\draw[thick, dashed,<->,blue!50] (-.4,7) --node[midway,above]{$\alpha
n$} (4.4,7);
\end{scope}
}

\filldraw[thick,fill=gray!15] (-12,-2.5) rectangle (31,5.5);
\draw[very thick] (-.3,.3) -- node[midway,below,align=center]{initial
configuration~$\eta$}
(19.3,.3);
\foreach \y in
{(1,1),(2,1),(2,2),(2,3),(4,1),(6,1),(6,2),(10,1),(11,1),(12,1),(14,1),(14,2),(18,1),(18,2)}
{
\filldraw \y circle(.4);
}

\draw[very thick,dotted,blue!50, line cap=round] (-.6,.3)
-- node[below,align=center]{empty}(-10.4,.3) (19.6,.3) -- node[below,align=center]{empty}(29.4,.3);

\draw[blue!50,decorate,very thick] (-5.4,-3.5) --node[below=3pt]{sleeps ignored}
(-.6,-3.5);
\draw[blue!50,decorate,very thick] (19.6,-3.5) --node[below=3pt]{sleeps ignored}
(24.4,-3.5);

\draw[ultra thick, red!50!black, line cap=round, ->] (32,-2)
to[bend left] node[right,align=left]
{\textbf{Step 1:}\\
Stabilize~$V_{n+2\alpha n}$ with no particle\\
allowed to sleep in~$V_{n+2\alpha n}\setminus V_n$.}
(32,-8);

\begin{scope}[shift={(0,-15)}]
\filldraw[thick,fill=gray!15] (-12,-1.5) rectangle (31,7.5);
\draw[very thick, green!50!black] (-.3,.3) -- (19.3,.3);
\foreach \y in {(3,1),(4,1),(6,1),(9,1),(13,1),(18,1)}
{
\draw[green!50!black] \y node{$\s$};
}
\draw[very thick,dotted,blue!50, line cap=round] (-.6,.3)
-- (-10.4,.3) (19.6,.3) --
 (29.4,.3);

\begin{scope}[red!50!black]
\foreach \x in {1,...,5}
{
\fill (-6,\x) circle(.4);
}
\foreach \x in {1,...,3}
{
\fill (25,\x) circle(.4);
}

\draw[very thick,->] (0,3) to[bend right] (-4.7,4);
\draw[very thick,->] (19,2.5) to[bend left] (23.7,3.5);
\draw (9.5,5.5) node{$M'_n$ particles jumped out
of~$V_{n+2\alpha n}$.};
\end{scope}
\end{scope}

\draw[ultra thick, green!50!black, line cap=round, ->] (32,-16)
to[bend left] node[right,align=left]
{\textbf{Step 2:}\\
Try to stabilize~$V_{n+4\alpha n}\setminus V_n$,\\
hoping that no particle comes\\
back in~$V_n$ or exits~$V_{n+4\alpha n}$.}
(32,-22);

\draw[blue!50,decorate,very thick] (-10.4,-17.5) --node[below=3pt]{sleeps
allowed}
(-.6,-17.5);
\draw[blue!50,decorate,very thick] (19.6,-17.5) --node[below=3pt]{sleeps allowed}
(29.4,-17.5);

\begin{scope}[shift={(0,-29)}]
\filldraw[thick,fill=gray!15] (-12,-1.5) rectangle (31,7.5);

\begin{scope}[green!50!black]
\draw[very thick]
(-.3,.3) -- (19.3,.3)
(-.6,.3) -- (-10.4,.3)
(19.6,.3) -- (29.4,.3);

\foreach \y in {-10,-8,-4,-2,-1,3,4,6,9,13,18,21,28,29}
{
\draw (\y,1) node{$\s$};
}

\draw[very thick,->] (-6.4,4) to[bend right] (-9.7,1.5);
\draw[very thick,->] (-6.4,1) to[bend right] (-7.7,1.5);
\draw[very thick,->] (-5.6,5) to[bend left] (-1.3,1.5);
\draw[very thick,->] (-5.6,3) to[bend left] (-2.3,1.5);
\draw[very thick,->] (-5.6,2) to[bend left] (-4.3,1.5);

\draw[very thick,->] (24.6,2) to[bend right] (21.3,1.5);
\draw[very thick,->] (25.4,3) to[bend left] (28.7,1.5);
\draw[very thick,->] (25.4,1) to[bend left] (27.7,1.5);

\draw (9.5,5) node
{$\eta$ is stabilized in~$V_{n+4\alpha n}$.};
\end{scope}

\foreach \x in {1,...,5}
{
\draw[gray,densely dotted,thick] (-6,\x) circle(.4);
}
\foreach \x in {1,...,3}
{
\draw[gray,densely dotted,thick] (25,\x) circle(.4);
}
\end{scope}

\end{tikzpicture}
\end{center}
\caption{
\label{fig_thm_explicit}
The two steps of the proof of Theorem~\ref{thm_explicit}.}
\end{figure}

During the first step, we stabilize~$V_{n+2\alpha n}$, forbidding
particles to fall asleep out of~$V_n$, and freezing particles once
they jump out of~$V_{n+2\alpha n}$.
The number of particles which jump out of~$V_{n+2\alpha n}$ during
this step, which we
denote by~$M'_n$, may differ from~$M_n$, which is the number of
particles jumping out of~$V_n$ when we just stabilize~$V_n$.
Yet, it turns out that~$M'_n$ is stochastically
dominated by~$M_n$.
This is the content of Lemma~\ref{lemma_NML}:
adding these no man's lands around~$V_n$ does not increase, in
distribution, the number of particles which jump out.

Then, during the second step, we try to
stabilize these~$M'_n$ particles in~$V_{n+4\alpha n}\setminus V_n$.
This step is said to be successful if, doing so, no
particles jump out of~$V_{n+4\alpha n}$ or come back
in~$V_n$.
In this case, we managed to stabilize the configuration~$\eta$
in~$V_{n+4\alpha n}$ with no
particle jumping out of~$V_{n+4\alpha n}$, forbidding some
particles to fall asleep at some stages.
By the monotonicity property with respect to enforced activation, this
implies that no particle would have jumped out of~$V_{n+4\alpha n}$
also if we had not forbidden these particles to fall asleep, that is
to say,~$M_{n+4\alpha n}=0$.

Note that the overall density in the enlarged segment
with the empty spaces around is~$\zeta/(1+4\alpha)$.
Thus, if~$\alpha$ is chosen small enough so
that~$\zeta/(1+4\alpha)>\zeta_c$, then Theorem~\ref{thm_no_exit} gives
an upper bound on the probability that~$M_{n+4\alpha n}=0$.

Then, if we consider~$\varepsilon>0$ small enough so
that~$\varepsilon/\alpha<\lambda/(1+\lambda)$, then we can show that
the second stage succeeds with high probability, conditioned on the
event that~$M'_n\leq\varepsilon n$ (to show this we adapt the trapping
procedure which is used in~\cite{RS12} to obtain Theorem~\ref{lemmaRS}).
Thus, we can translate the upper bound on~$\PP_\eta\big(M_{n+4\alpha n}=0\big)$ coming from Theorem~\ref{thm_no_exit} into an
upper bound on~$\PP_\eta(M'_n\leq \varepsilon n)$.
The stochastic domination given by Lemma~\ref{lemma_NML} then allows
us to translate this into the claimed upper bound
on~$\PP_\eta(M_n\leq\varepsilon n)$.

The proof of Theorem~\ref{thm_explicit} is the object of
Section~\ref{section_explicit}.

\subsection{Obtaining Theorem~\ref{main_thm} and
Proposition~\ref{fact_critical}}

Given Theorem~\ref{thm_explicit}, our Theorem~\ref{main_thm} easily
follows. 
The only detail is that, to show that we have indeed an
equivalence, there remains to show Proposition~\ref{fact_critical}, which
states that, at criticality, there is not a positive fraction which
jumps out of the box.
This is presented in Section~\ref{section_main}.

\subsection{Higher dimension or longer jumps}
\label{section-high-dim}

As explained before, our proof strategy beaks down in higher dimension
or when particles can make longer jumps.
More precisely, there are two stages where we crucially rely on the
fact that we are in dimension~1 and that jumps are limited to the
nearest neighbours.

First, we use these two assumptions in the construction of the
coarse-graining coupling given by Proposition~\ref{prop_coupling}
when, after the bad stabilization of a block, we force the particles
to move to form a new copy of the configuration~$\eta$ on the
following block in the direction of the origin.
This can be done without
disturbing the sleeping particles which are on other blocks closer to
the origin, but only because we are in dimension~1 with
nearest-neighbour jumps (otherwise the particles could go around or
skip the site where we want to bring them).

These two assumptions are also used in the proof of
Lemma~\ref{lemma_NML}, which shows that adding no man's lands
around a segment does not increase, in distribution, the number of
particles which jump out.
Indeed, the proof of this Lemma uses the fact that, if you stabilize
a segment by always
toppling the leftmost active site, then each time that a particle
jumps out by the left exit, all the other particles must be active.
In higher dimension or with longer jumps this would no longer be the
case, 
but it would be interesting to investigate whether or not a result
similar to Lemma~\ref{lemma_NML} would still hold.

\section{The site-wise representation of the model}
\label{section_tools}

We now describe the site-wise representation of the model, with an
array of sleep and jump instructions above the sites.
We refer to the survey~\cite{Rolla20} for a more detailed
presentation.

\subsection{Topplings and odometers}
\label{site_wise}

A crucial ingredient in the study of Activated Random Walks is the
site-wise representation, also known as Diaconis-Fulton
representation~\cite{DF91,RS12}.
Let~$\eta:\Z^d\to\N\cup\{\s\}$ be a fixed initial
configuration, and let us consider a fixed array of
instructions~$\tau=(\tau_{x,i})_{x\in\Z^d,\,i\geq 1}$ where, for
every~$x\in\Z^d$ and every~$i\geq 1$, the instruction~$\tau_{x,i}$
can be either a sleep instruction or a jump instruction to some
site~$y\in\Z^d$.

The idea is that, once this initial configuration and this array of
instructions are fixed, the evolution of the system can be constructed by
looking at these instructions each time that something happens at some
site.
As we use instructions of the array, we keep track of which
instructions have already been used, with the help of a function called the
odometer~$h:\Z^d\to\N$, which counts, at each site, how many
instructions have already been used.

When we use an instruction at a site~$x$, we say that we topple~$x$.
For a given fixed configuration~$\eta$, we
say that it is legal (respectively, acceptable) for~$\eta$ to
topple a site~$x$ if~$x$ contains at least one active
particle (respectively, at least one particle) in~$\eta$.

If a toppling is legal or acceptable, then this toppling consists in
using the next
instruction~$\tau_{x,\,h(x)+1}$ to update the
configuration~$\eta$: if this instruction is a sleep instruction, then
the particle at~$x$ falls asleep if it is alone (whereas nothing happens
if there are at least two particles at~$x$), and if it is a jump
instruction to another site~$y$, one particle at~$x$ jumps to site~$y$, waking up the
sleeping particle there if there is one.
If the toppling was only acceptable but not legal, we first wake up
the particle at~$x$ before applying the toppling.
The resulting configuration is denoted by~$\tau_{x,\,h(x)+1}\eta$.
Thus, for a fixed realization of the array~$\tau$, the toppling at a
site~$x$ consists of an operator
$$\Phi_x^\tau\ :\ (\eta,\,h)\ \longmapsto\
(\tau_{x,\,h(x)+1}\eta,\,h+\delta_x)\,,$$
which is only defined if the toppling is acceptable.

If~$\alpha=(x_1,\,\ldots,\,x_k)$ is a sequence of sites
of~$\Z^d$, we say that the toppling sequence~$\alpha$ is~$\tau$-legal
(resp.,~$\tau$-acceptable)
for~$(\eta,\,h)$ if for every~$i\in\{1,\,\ldots,\,k\}$, it is legal
(resp., acceptable)
for~\smash{$\Phi_{x_{i-1}}^\tau\circ\cdots\circ\Phi_{x_2}^\tau\circ\Phi_{x_1}^\tau(\eta,\,h)$} to
topple~$x_i$, that is to say, if the configuration resulting from the
first~$i-1$ topplings has at least one active particle (resp., at
least one particle) on the site~$x$.
If~$\alpha$ is acceptable, applying the toppling sequence~$\alpha$
means
applying~\smash{$\Phi^\tau_\alpha=\Phi^\tau_{x_k}\circ\cdots\circ\Phi^\tau_{x_1}$}.
We define the odometer of a toppling sequence~$\alpha$
as~$m_\alpha=\delta_{x_1}+\cdots+\delta_{x_k}$, which simply counts
how many times each site appears in the sequence~$\alpha$.
We also define, for every~$V\subset\Z^d$,
\begin{equation}
\label{def_mVeta}
m_{V,\,\eta}^\tau\ =\ \sup_{\alpha\subset V,\ \alpha\text{
is~$\tau$-legal for }\eta}
m_\alpha\,,
\end{equation}
where the notation~$\alpha\subset V$ means that all
the sites appearing in~$\alpha$ must belong to~$V$.
The total stabilization odometer associated with the
configuration~$\eta$ is defined as:
\begin{equation}
\label{def_m_eta}
m_{\eta}^\tau
\ =\ 
\sup_{\alpha\text{ is~$\tau$-legal for }\eta}
m_\alpha
\ =\ 
\sup_{V\subset\Z^d}
m_{V,\,\eta}^\tau
\,.
\end{equation}

\subsection{Abelian property and monotonicity}

An important advantage of the site-wise construction is the following
property, which states that the order with which we perform the
topplings is irrelevant, allowing us to use whatever convenient
strategy to choose which sites to topple. 
We say that a sequence of topplings stabilizes~$\eta$ in~$V$
if the configuration resulting from the application of the toppling
sequence is stable in~$V$, meaning that there are no active particles
in~$V$.

\begin{lemma}[Abelian property, Lemma~2 in~\cite{RS12}]
\label{lemma_abelian}
If~$\alpha$ and~$\beta$ are both legal toppling sequences for~$\eta$
that are contained in~$V$ and stabilize~$\eta$ in~$V$,
then~$m_\alpha=m_\beta=m_{V,\,\eta}^\tau$ and the resulting configurations
are equal, that is to
say,~$\Phi^\tau_\alpha(\eta,\,0)=\Phi^\tau_\beta(\eta,\,0)$.
\end{lemma}

We also have the following monotonicity property, which shows that
acceptable topplings may be used when looking for upper
bounds on the legal odometer:

\begin{lemma}[Lemma~2.1 in~\cite{Rolla20}]
\label{lemma_monotonicity}
If~$\alpha$ is an acceptable sequence of topplings that
stabilizes~$\eta$ in~$V$, and~$\beta\subset V$ is a legal sequence of
topplings for~$\eta$, then~$m_\alpha\geq m_\beta$.
Thus, if~$\alpha$ is an acceptable sequence of topplings that
stabilizes~$\eta$ in~$V$, then~$m_\alpha\geq m_{V,\,\eta}^\tau$.
\end{lemma}

\subsection{Probability of fixation}

We now make~$\eta$ and~$\tau$ become random.
Let~$\lambda>0$, let~$P:\Z^d\times\Z^d\to[0,1]$ be a jump
distribution and let~$\nu$ be a probability distribution on the
set~\smash{$(\N\cup\{\s\})^{\Z^d}$} of all possible initial
configurations.
We then write~$\PP^\nu_{\lambda,P}$ for the measure
relative to the activated random walk model with sleep rate~$\lambda$,
jump distribution~$P$ and initial configuration
distributed according to~$\nu$.
Then, the probability of fixation of the model is related to the
stabilization odometer through (see~\cite{RS12})
\begin{equation}
\label{fixation_odometer}
\PP^\nu_{\lambda,P}\big(\text{the system fixates}\big)
\ =\ 
\nu\otimes\mathcal{P}_{\lambda,P}\big(m_\eta^\tau(0)<\infty\big)
\,,
\end{equation}
where~$\mathcal{P}_{\lambda,P}$ is the measure on all the possible
stacks of instructions~$(\tau_{x,i})_{x\in\Z^d,\,i\geq
1}$ such that the instructions are independent and for
every~$x\in\Z^d$ and~$i\geq 1$, the instruction~$\tau_{x,i}$ is a
sleep instruction with probability~\smash{$\lambda/(1+\lambda)$} and it is a
jump instruction to~$y\in\Z^d$ with
probability~$P(x,\,y)/(1+\lambda)$.
Thus, to know whether the system fixates or not, it is enough to look
at this array of instructions~$\tau$ and to determine
whether the stabilization odometer at the origin,~$m_\eta^\tau(0)$,
is finite or not.

\section{Proof of Theorem~\ref{thm_no_exit}: the probability that no
particle exits}
\label{section_no_exit}

This section is devoted to the proof of Theorem~\ref{thm_no_exit},
following the strategy described in Section~\ref{section_sketches}.

\subsection{The block configuration}
\label{section_block_config}

We now construct the random block configuration.
Let~$n\geq 1$,~\smash{$\eta:V_n\to\N$} and let~$\nu$ be a
probability distribution on~$\N^\Z$.
For every~$k\in\Z$ we define~$B_k=V_n+kn$, called the block
number~$k$.
Let~$\bar\eta$ be the configuration which contains a copy
of~$\eta$ inside each of these blocks, i.e., for every~$k\in\Z$ and
every~$x\in B_k$ we have~$\bar\eta(x)=\eta(x-kn)$.
Then, we consider the
application~\smash{$\varphi_\eta:\N^\Z\to\N^\Z$} defined by
\[
\forall \xi\in\N^\Z\quad
\forall k\in\Z\quad
\forall x\in B_k\qquad
\varphi_\eta(\xi)(x)
\ =\ 
\xi(k)\times \bar\eta(x)
\,.
\]
With this notation, we
define~$\nu_{\eta}=(\varphi_\eta)_\star\nu$, the
push-forward measure of~$\nu$ through this map~$\varphi_\eta$.
For example if under~$\nu$, the variables~$(\xi(k))_{k\in\Z}$ are
i.i.d.\ Bernoulli with a certain parameter~$q$, then~$\nu_\eta$ is
the measure on the initial configurations such that each of the
blocks~$B_k$ for~$k\in\Z$ contains a translated copy
of~$\eta$ with probability~$q$, independently for each block.

\subsection{The coarse-graining coupling}

We now formalize the coupling that will allow us to study the model
started with this block configuration.
This coupling is illustrated in Figure~\ref{fig_coupling}.
 
\begin{proposition}
\label{prop_coupling}
Let~$\lambda>0$, let~$P:\Z\times\Z\to[0,1]$ be a translation-invariant
nearest-neighbour jump distribution,
let~$n\geq 1$, let~\smash{$\eta:V_n\to\N$} with~$\eta\neq 0$ and
let~$p=\PP_\eta(M_n=0)$.
Let~$\lambda'=p/(1-p)$ and let~$D:\Z\times\Z\to[0,1]$ be a
nearest-neighbour jump
distribution with directed jumps towards the origin, that is to say, such that for every~$x\geq 1$, we have~$D(x,\,x-1)=1$
and~$D(-x,\,-x+1)=1$.

Then, for every probability distribution~$\nu$ on~$\N^\Z$, there exists a coupling
between~$\nu_\eta\otimes\mathcal{P}_{\lambda,P}$ (called the block
model)
and~$\nu\otimes\mathcal{P}_{\lambda'\!,D}$ (called the coarse-grained
model) such that, if~$(\xi,\,\tau)$
and~$(\xi',\,\tau')$ are coupled with these two respective
distributions then we have the implication
\begin{equation}
\label{implication_coupling}
\big\{m_{\xi'}^{\tau'}(0)=0\big\}
\ \subset\ 
\big\{\forall x\in B_0\,,\ m_\xi^\tau(x)=0\big\}
\,,
\end{equation}
that is to say, if the origin is never visited in the coarse-grained
model then the block of the origin is never visited in the block
model.
\end{proposition}

\begin{proof}
Let~$\lambda,\,P,\,n,\,\eta,\,p,\,\lambda',\,D$ and~$\nu$ be as in the
statement.
Let~$\xi'$ be a random initial configuration with
distribution~$\nu$, and let~$\xi=\varphi_\eta(\xi')$ be the block
configuration obtained from~$\xi'$, as defined in
Section~\ref{section_block_config}, so that~$\xi$ has
distribution~$\nu_{\eta}$.
Let~$\tau$ be an array of instructions with
distribution~$\mathcal{P}_{\lambda,P}$, independent of~$\xi'$ (and
hence also independent of~$\xi$).

There now remains to define the array~$\tau'$.
The idea is to stabilize the configuration~$\xi$ using the
instructions in~$\tau$, one block after
another, writing instructions in~$\tau'$ along the procedure, depending on
whether we get good or bad stabilizations.
If the block of the origin is visited by a particle then the procedure
stops, because we only care about the event that the block of the
origin is never visited.

Formally, we construct recursively, for each~$j\in\N$, an array of
instructions~$\tau'_j$ and two sequences of topplings~$\alpha_j$
and~$\alpha'_j$ such that, for every~$j\in\N$, if we denote
respectively by~$\beta_j$ the concatenation of the
sequences~$\alpha_0,\,\ldots,\,\alpha_j$ and by~$\beta'_j$ the
concatenation of the sequences~$\alpha'_0,\,\ldots,\,\alpha'_j$, the
following properties are satisfied for every~$j\in\N$
(with~\eqref{priority} and~\eqref{modifs} concerning only~$j\geq 1$):

\begin{enumerate}[(i)]
\item
\label{acceptable}
The sequence~$\beta_j$ is~$\tau$-acceptable for~$(\xi,\,0)$,
allowing us to define~$(\xi_j,\,h_j)=\Phi^\tau_{\beta_j}(\xi,\,0)$.

\item
\label{legal}
The sequence~$\beta'_j$ is~$\tau'_j$-legal for~$(\xi',\,0)$, allowing
us to
define~\smash{$(\xi'_j,\,h'_j)=\Phi^{\tau'_j}_{\beta'_j}(\xi',\,0)$}.

\item
\label{coherent}
The value
of the coarse-grained configuration~$\xi'_j$ at every site~$k\in\Z$ is related
to the
configuration~$\xi_j$ inside the block~$B_k$ through:
\[
\left\{
\begin{aligned}
&\xi'_j(k)=0
\quad\Longrightarrow\quad
\xi_j\text{ empty in~$B_k$,}
\\
&\xi'_j(k)=\s
\quad\Longrightarrow\quad
\xi_j\text{ stable in~$B_k$, with~$\|\eta\|$ sleeping particles
in~$B_k$,}
\\
&\xi'_j(k)=\ell\in\N\setminus\{0\}
\quad\Longrightarrow\quad
\xi_j\text{ contains~$\ell\times\|\eta\|$ particles in~$B_k$, including a copy of~$\eta$,}
\end{aligned}
\right.
\]
where in the last case we mean that~$\sum_{x\in
B_k}|\xi_j(x)|\geq\ell\,\|\eta\|$ and~$\xi_j\geq\bar\eta$ in~$B_k$,
where~$\bar\eta$ is the configuration
defined in Section~\ref{section_block_config}, so that in the
case~$\ell=1$ this reduces to~$\xi_j=\bar\eta$ in~$B_k$.

\item
\label{sleeps_gathered}
Active particles remain farther away from the
origin than sleeping particles, in the following sense: for
every~$k,\,\ell\in\N$,
\[
\Big[
\big(\xi'_j(k)= 1
\ \text{ and }\ 
\xi'_j(\ell)=\s\big)
\quad\text{or}\quad
\big(\xi'_j(-k)= 1
\ \text{ and }\ 
\xi'_j(-\ell)=\s\big)
\Big]
\quad\Longrightarrow\quad
\ell\leq k
\,.
\]

\item
\label{priority}
We always topple in priority a block as close as possible to the
origin, that is to say, for every~$k\geq 1$,
if the set~$\{-k,\,\ldots,\,k\}$
is unstable in~$\xi'_{j-1}$ and~$\xi'_{j-1}(0)=0$, then~$\alpha'_j$
topples at least one site in this set.

\item
\label{modifs}
The instructions in the arrays~$\tau'_j$ are modified only when used
by~$\alpha'_j$,
that is to say, for every~$x\in\Z$ and every~$i\geq 1$,
if~$(\tau'_j)_{x,i}\neq(\tau'_{j-1})_{x,i}$ then~$h'_{j-1}(x)<i\leq
h'_j(x)$.

\item
\label{clear_origin}
As long as the origin remains empty
in the coarse-grained configuration~$\xi'_j$, the block of the
origin~$B_0$ is not toppled by the sequence~$\beta_j$, that is to
say,
\[
\xi'_j(0)=0
\quad\Longrightarrow\quad
\big(
\forall x\in B_0\quad
h_j(x)=0
\big)
\,.
\]

\item
\label{indep_tau}
The array~$\tau'_j$ has distribution~$\mathcal{P}_{\lambda'\!,D}$
and is independent of~$\xi'$.
\end{enumerate}

For~$j=0$, we simply take~$\tau'_0$ to be an array of
instructions with distribution~$\mathcal{P}_{\lambda'\!,D}$, independent
of~$(\xi,\,\xi',\,\tau)$, and we let~$\alpha_0$ and~$\alpha'_0$ be
empty toppling sequences, so that~$\xi_0=\xi$ and~$\xi'_0=\xi'$.
Thus, the items~\eqref{coherent} and~\eqref{sleeps_gathered} hold by definition
of~$\xi=\varphi_\eta(\xi')$ and because
there is no sleeping particle
in~$\xi'$, and the other items hold
trivially.

Let~$j\in\N$, assume that the array~$\tau'_j$ and the
sequences~$\alpha_0,\,\ldots,\,\alpha_j$
and~$\alpha'_0,\,\ldots,\,\alpha'_j$ are
already constructed and satisfy the eight above properties, and let us
construct~$\tau'_{j+1},\,\alpha_{j+1}$ and~$\alpha'_{j+1}$ which
satisfy the properties at rank~$j+1$.

If the configuration~$\xi'_j$ is stable or such that~$\xi'_j(0)\neq
0$ then we
simply take~$\alpha_{j+1}$ and~$\alpha'_{j+1}$ to be empty toppling
sequences, and we let~$\tau'_{j+1}=\tau'_j$, and the
conditions~\eqref{acceptable} to~\eqref{clear_origin}
are inherited at rank~$j+1$.
We will deal with condition~\eqref{indep_tau}
afterwards, once we defined~$\tau'_{j+1}$ in the different cases.

Assume now that~$\xi'_j$ is unstable and~$\xi'_j(0)=0$.
Let~$k\in\Z$ be such that~$\xi'_j(k)\geq 1$, with~$|k|$ being
minimal among such sites (with an arbitrary deterministic rule to
break the tie if any).
Since~$\xi'_j(0)=0$, we have~$k\neq 0$.
Let us define~$\ell=k-1$ if~$k>0$ and~$\ell=k+1$ if~$k<0$, so
that~$D(k,\,\ell)=1$.

We now distinguish between two cases.
First, let us assume that~$\xi'_j(k)=1$.
Then, item~\eqref{coherent} tells us that in the block~$B_k$ the
configuration~$\xi_j$ is a copy of~$\eta$.
We consider~$\gamma$ a~$\tau$-legal
toppling sequence
for~$(\xi_j,\,h_j)$ which topples the sites of~$B_k$ (and only these
sites) until all sites of~$B_k$
are stable (with probability~$1$ such a sequence exists, let us assume
that this is the case).

By definition of~$p$ as the
probability of a good stabilization, with probability~$p$ no particles
jump out of~$B_k$ while performing this sequence of topplings~$\gamma$.
In this case, we let~$\tau'_{j+1}$ be the
array obtained from~$\tau'_j$ by replacing the instruction
number~$h'_j(k)+1$ at site~$k$ with a sleep instruction, and we
define~$\alpha_{j+1}=\gamma$ and~$\alpha'_{j+1}=(k)$, so that
by construction items~\eqref{acceptable} and~\eqref{legal} hold for~$j+1$.
We then have~$\xi'_{j+1}(k)=\s$ and~$\xi_{j+1}$ is
stable in the block~$B_k$, with nothing changed out of this block
compared to~$\xi_j$,
which shows that item~\eqref{coherent} still holds
at rank~$j+1$.
Besides, the minimality of~$|k|$ ensures that
items~\eqref{sleeps_gathered} and~\eqref{priority}
remain satisfied at rank~$j+1$.
Property~\eqref{modifs} holds for~$j+1$ because~$\alpha'_{j+1}$ uses
precisely the instruction that we replaced.
Lastly, item~\eqref{clear_origin} is inherited
because we performed topplings only in the block~$B_k$, with~$k\neq
0$.

Assume now that a bad stabilization happens: a particle jumps out
of~$B_k$ while applying~$\gamma$.
We then let~$\tau'_{j+1}$ be the
array obtained from~$\tau'_j$ by replacing the instruction
number~$h'_j(k)+1$ at site~$k$ with a jump instruction towards the origin.
Then, we force
these~$\|\eta\|$ particles which were in the block~$B_k$ in the
configuration~$\xi_j$ to walk, with topplings which
are~$\tau$-acceptable for~$\Phi_{\gamma}^\tau(\xi_j,\,h_j)$, until these
particles form a copy of~$\eta$ in the block~$B_\ell$ (recall
that~$B_\ell$
is the neighbour block of~$B_k$ in the direction of the origin).
We do this by performing only topplings on the sites of the
blocks~$B_m$ with~$m\geq\ell$ if~$k>0$ or on the sites of the blocks~$B_m$
with~$m\leq\ell$ if~$k<0$ (to see this, label the particles, assign
to each particle a destination site in~$B_\ell$ and move each particle
one by one until it reaches its destination, which eventually happens
with probability~$1$, and doing so no particle walks closer to the
origin than its destination, because the jumps are only to the nearest
neighbours).
We call~$\delta$ the obtained acceptable sequence, and
we then let~$\alpha_{j+1}=(\gamma,\,\delta)$
and~$\alpha'_{j+1}=(k)$.
Here also, items~\eqref{acceptable} and~\eqref{legal} hold by
construction.
Item~\eqref{sleeps_gathered}, together with item~\eqref{coherent},
ensures that none of these blocks on which we perform topplings
can contain sleeping particles, except maybe the block~$B_\ell$.
This entails that items~\eqref{coherent} and~\eqref{sleeps_gathered} still
hold at rank~$j+1$.
As in the previous case, property~\eqref{priority} holds for~$j+1$ by
minimality of~$|k|$ and property~\eqref{modifs} because~$\alpha'_{j+1}$
uses the replaced instruction.
The property~\eqref{clear_origin} is also inherited because
if~$\ell\neq 0$ then we did not topple the block~$B_0$, whereas
if~$\ell=0$ then~$\xi'_{j+1}(0)\neq 0$.

We now deal with the case~$\xi'_j(k)\geq 2$.
Then, item~\eqref{coherent} tells us
that~$\xi_j\geq\bar\eta$ in~$B_k$ and that the
configuration~$\xi_j-\bar\eta$ contains at least~$\|\eta\|$
particles.
We move~$\|\eta\|$ of these particles with topplings read from~$\tau$
and acceptable for~$(\xi_j,\,h_j)$ until these particles form a copy
of~$\eta$ in the block~$B_\ell$, performing topplings only on the
sites of the blocks~$B_m$ with~$m\geq\ell$ if~$k>0$ or~$m\leq\ell$
if~$k<0$.
It is important that we move these extra particles, leaving
aside~$\|\eta\|$ other particles which already form a copy of~$\eta$ in the
block~$B_k$, so that the configuration will still contain a copy
of~$\eta$ in this block~$B_k$ at rank~$j+1$.
We then let~$\alpha_{j+1}$ be the acceptable sequence of topplings
performed.
Here we do not need to modify the array~$\tau'_j$ because,
since~$\xi'_j(k)\geq 2$, we can move one particle from~$k$ to~$\ell$
with legal topplings read from~$\tau'_j$, by simply reading
instructions until we find a jump instruction (and this jump
instruction necessarily points to~$\ell$, by definition of the
directed jump distribution~$D$).
Hence, we let~$\tau'_{j+1}=\tau'_j$ and we let~$\alpha'_{j+1}$ be the
sequence consisting of~$m_0$ occurrences
of~$k$, where~\smash{$m_0=\inf\{m\geq
1\,:\,(\tau'_j)_{k,\,h'_j(k)+m}\neq\s\}$} (which is almost surely
finite, let us assume that it is the case).
As before, items~\eqref{acceptable} to~\eqref{clear_origin} are
inherited at rank~$j+1$.

We constructed~$\alpha_{j+1}$,~$\alpha'_{j+1}$
and~$\tau'_{j+1}$ in the different cases, and there now remains to
check that item~\eqref{indep_tau} remains true at rank~$j+1$.
To show this, consider
the~$\sigma$-field
\[
\mathcal{F}_j
\ =\ 
\sigma\big(\xi',\,
h_j,\,((\tau_j)_{x,i})_{x\in\Z,\,i\leq h_j(x)},\,
h'_j,\,((\tau'_j)_{x,i})_{x\in\Z,\,i\leq h'_j(x)}\big)
\,,
\]
which contains
the information of the initial configuration and of the instructions
used by~$\beta_j$ and~$\beta'_j$.
Then,~$\xi_j$ and~$\xi'_j$ are~$\mathcal{F}_j$-measurable, so the
event~$A$ that~$\xi'_j$ is unstable, with~$\xi'_j(0)=0$ and that the
site~$k$ that we consider in the construction at step~$j$ is such
that~$\xi'_j(k)=1$, is~$\mathcal{F}_j$-measurable.
Recall that we replace an instruction in~$\tau'_{j+1}$ compared
to~$\tau'_j$ only on this event~$A$, and that the
instruction that we write is a sleep instruction if and only if we get
a good stabilization.
Yet, conditioned on~$\mathcal{F}_j$, on this event~$A$
the conditional probability to obtain a good stabilization is constant
and equal to~$p$,
because it depends only on instructions of~$\tau$ that were not used
by~$\beta_j$, and thus are not revealed yet.
Hence, when we rewrite an instruction, then the instruction
that we write is independent of~$\mathcal{F}_j$, as was
the instruction that was present at this position in~$\tau'_j$, and
moreover these two instructions have the same
distribution, since~$p=\lambda'/(1+\lambda')$ (the instruction that we
write is a sleep instruction with probability~$p$, while~$\lambda'/(1+\lambda')$ is the probability that a given instruction in~$\tau'_j$ is a sleep instruction).
Hence,~$\tau'_j$ and~$\tau'_{j+1}$ have the same conditional
distribution knowing~$\mathcal{F}_j$.
A fortiori, since~$\sigma(\xi')\subset\mathcal{F}_j$, we deduce
that~$\tau'_j$ and~$\tau'_{j+1}$ also have the same conditional
distribution knowing only~$\xi'$.
Yet, item~\eqref{indep_tau} at rank~$j$ ensures that~$\tau'_j$ is independent of~$\xi'$
and has distribution~$\mathcal{P}_{\lambda'\!,D}$: therefore, it
remains true at rank~$j+1$.

Thus, we constructed~$\tau'_j,\,\alpha_j$ and~$\alpha'_j$ for
every~$j\in\N$, which satisfy the eight above properties.
Note that property~\eqref{modifs} entails that for every fixed~$k\in\Z$
and~$i\geq 1$, in the sequence~$(\tau'_j)_{j\in\N}$ the instruction at
position~$(k,i)$ changes at most once, so that the
sequence~\smash{$((\tau'_j)_{k,i})_{j\in\N}$} is stationary.
Hence, the sequence of arrays~$(\tau'_j)_{j\in\N}$ weakly converges to
a limit that we denote~$\tau'$.
The property~\eqref{indep_tau} passes to the limit: this
array~$\tau'$ is also independent
of~$\xi'$ and has distribution~$\mathcal{P}_{\lambda'\!,D}$.

There now remains to show that the
implication~\eqref{implication_coupling} holds.
From now on and until the end of the proof, we assume
that~\smash{$m_{\xi'}^{\tau'}(0)=0$}.
Since by definition~$\xi'(0)\neq\s$, this implies
that~$\xi'(0)=0$, and that the site~$0$ remains empty
after applying any~$\tau'$-legal sequence of topplings.
Note that item~\eqref{modifs} ensures that for every~$j\in\N$ the
instructions of~$\tau'_j$ that are
used by~$\beta'_j$ are not modified in subsequent steps, and are
therefore identical in~$\tau'$.
Therefore, the
sequence~$\beta'_j$ is not only~$\tau'_j$-legal but also~$\tau'$-legal
for~$(\xi',\,0)$.
By what we just said, this implies that~$\xi'_j(0)=0$ for every~$j\in\N$.

This implies that for every~$k\in\Z\setminus\{0\}$, the
sequence~$\beta'_j$ performs at most~$|k|-1$ jumps at the site~$k$.
Indeed, if~$\beta'_j$ performs at least~$|k|$ jumps at
a site~$k\in\Z\setminus\{0\}$ then the resulting
configuration~$\xi'_j$
contains at least~$|k|$ particles strictly between the origin and~$k$
and then one of these particles will eventually reach the origin, due
to the priority rule of item~\eqref{priority},
which contradicts our finding that~$\xi'_j(0)=0$ for every~$j\in\N$.

Let now~$k\geq 1$,~$V'=\{-k,\,\ldots,\,k\}$ and~$V=\cup_{\ell\in
V'}B_\ell$ and let us show
that for every~$x\in B_0$,~$m_{\xi,V}^\tau(x)=0$.
For every~$j\in\N$, let~$u_j$ be the total number of jumps performed by the
sequence~$\alpha'_j$ on the sites of~$V'$.
The above paragraph ensures that this (non-decreasing) sequence~$(u_j)_{j\in\N}$ is
bounded, and therefore constant from a certain
rank~$j_0$, which means that after step~$j_0$ we do not perform
any more jumps on~$V'$.
Yet, for every~$j\geq j_0$ such that~$\xi'_j$ is still
unstable in~$V'$, the priority rule of item~\eqref{priority} forces us
to topple a site of~$V'$, but this toppling can only be a sleep
because~$j\geq j_0$, so that one particle of~$V'$ falls asleep.
Hence, at the latest at step~$j=j_0+2k$, the configuration~$\xi'_j$ must be
stable in~$V'$.

By virtue of the connection given by item~\eqref{coherent} between~$\xi_j$
and~$\xi'_j$, this implies that the
configuration~$\xi_j$ is stable in~$V$.
Thus,~$\beta_j$ is a~$\tau$-acceptable sequence which
stabilizes~$\xi$ in~$V$.
By Lemma~\ref{lemma_monotonicity}, this implies
that~$m_{V,\xi}^\tau\leq m_{\beta_j}=h_j$.

Yet, since~$\xi'_j(0)=0$ the property~\eqref{clear_origin}
entails that~$h_j(x)=0$ for every~$x\in B_0$, whence~$m_{V,\xi}^\tau(x)=0$ for
every~$x\in B_0$.
Since this holds for every~$k\geq 1$, we deduce
that for every~$x\in B_0$ we have~$m_\xi^\tau(x)=0$, which concludes
the proof of the implication~\eqref{implication_coupling}.
\end{proof}

\subsection{Positive probability that the block of the origin is never
visited}

Using the coupling given by Proposition~\ref{prop_coupling} and
combining it with Theorem~\ref{lemmaRS}, we obtain:

\begin{corollary}
\label{cor}
Let~$\lambda>0$, let~$P$ be a translation-invariant nearest-neighbour
jump distribution,
let~$n\geq 1$, let~$\eta:V_n\to\N$ with~$\eta\neq 0$ and
let~$p=\PP_\eta(M_n=0)$.
Then for every i.i.d.\ initial distribution~$\nu$ with mean density of
particles~$q<p$ we have
\[
\nu_\eta\otimes\mathcal{P}_{\lambda,P}
\big(\forall x\in B_0\,,\ m_\xi^\tau(x)=0\big)
\ >\ 0
\,.
\]
\end{corollary}

\begin{proof}
Proposition~\ref{prop_coupling} shows that
\begin{equation}
\label{res_coupling}
\nu_\eta\otimes\mathcal{P}_{\lambda,P}
\big(\forall x\in B_0\,,\ m_\xi^\tau(x)=0\big)
\ \geq\ 
\nu\otimes\mathcal{P}_{\lambda'\!,D}
\big(
m_{\xi}^{\tau}(0)=0\big)
\,,
\end{equation}
with~$\lambda'=p/(1-p)$ and~$D$ a directed jump distribution, as defined
in the statement of Proposition~\ref{prop_coupling}.
Yet, since~$q<p=\lambda'/(1+\lambda')$, Theorem~\ref{lemmaRS} entails
that the right-hand side of~\eqref{res_coupling}
is strictly positive, whence the result.
\end{proof}

\subsection{Making the initial configuration translation-invariant}
\label{section_offset}

Now, we would like to deduce that for~$\nu$ and~$\nu_\eta$ as above,
the mean density of particles
in this initial distribution~$\nu_\eta$ is at most~$\zeta_c$,
using Theorem~\ref{thm_phase_transition}.
But there remains to deal with a small issue: this
distribution is not translation-invariant.
To obtain an initial distribution that is invariant by
translation, we simply apply a translation by a random offset.
Thus, we take~$\xi$ distributed according to~$\nu_\eta$ and~$Y$ a
uniform variable in~$B_0=V_n$, independent of~$\xi$, and we
define the configuration~$\xi^\text{inv}$ by writing, for
every~$x\in\Z$,
$$\xi^\text{inv}(x)
\ =\ 
\xi(x+Y)
\,.
$$
We call~$\nu_\eta^\text{inv}$ the distribution of~$\xi^\text{inv}$, and we
now check that this distribution is translation-ergodic.
First, by construction it is translation-invariant.
There remains to see that every translation-invariant event has
probability~$0$ or~$1$.
Let~$\xi'$ be a
random configuration with distribution~$\nu$ and
let~\smash{$\xi=\varphi_\eta(\xi')$}, so that~$\xi$ has
distribution~$\nu_{\eta}$, and let~$\xi^\text{inv}$ be defined as
above, with~$Y$ independent of~$\xi$. 
Then, if~$A$ is a Borel set of~$(\N\cup\{\s\})^\Z$ which is
invariant by translation, we can write
\begin{align*}
\nu_\eta^\text{inv}(A)
&\ =\ 
\PP(\xi^\text{inv}\in A)
\ =\ 
\sum_{y\in B_0}
\PP\big(\{Y=y\}\cap\{(x\mapsto\xi(x+y))\in A\}\big)
\ =\ 
\sum_{y\in B_0}
\PP(Y=y)\,\PP(\xi\in A)
\\
&\ =\ 
\PP(\xi\in A)
\ =\ 
\PP\big(\xi'\in(\varphi_\eta)^{-1}(A)\big)
\ =\ 
\nu\big((\varphi_\eta)^{-1}(A)\big)
\end{align*}
where, in the third equality we used the independence of~$Y$
and~$\xi$ and the fact that~$A$ is invariant by translation.
Yet, the distribution~$\nu$ is translation-ergodic because it is
i.i.d., and the event~$(\varphi_\eta)^{-1}(A)$ is invariant by
translation because~$A$ itself is invariant by translation, so we
have~$\nu_q\big((\varphi_\eta)^{-1}(A)\big)\in\{0,1\}$, which allows
us to deduce that~$\nu_\eta^\text{inv}(A)\in\{0,1\}$.
Hence, the distribution~$\nu_\eta^\text{inv}$ is translation-ergodic.

Lastly, note that the density of particles in this random initial
configuration is given by
\begin{equation}
\label{density_eta0}
\E\big[\xi^\text{inv}(0)\big]
\ =\ 
\E\big[\xi(Y)\big]
\ =\ 
\E\big[\xi'(0)\eta(Y)\big]
\ =\ 
q\times
\E\big[\eta(Y)\big]
\ =\ 
q\times\frac{\|\eta\|}{n}
\,.
\end{equation}

\subsection{Conclusion}

We now put the pieces together to obtain Theorem~\ref{thm_no_exit}.

\begin{proof}[Proof of Theorem~\ref{thm_no_exit}]
Let~$\lambda,\,P,\,\zeta,\,n$ and~$\eta$ be as in the statement, let us
write~$p=\PP_\eta(M_n=0)$ and consider~$q\in[0,\,p)$.

Let~$\nu$ be an i.i.d.\ probability distribution on~$\N^\Z$ with mean
density~$q$.
With the initial distribution~$\nu_\eta^\text{inv}$ defined in
Section~\ref{section_offset}, the relation~\eqref{fixation_odometer}
between fixation and the odometer allows us to write
\begin{align*}
\smash{\PP_{\lambda,P}^{\nu_\eta^\text{inv}}}\big(\text{the system fixates}\big)
&\ =\ 
\nu_\eta^\text{inv}\otimes\mathcal{P}_{\lambda,P}
\big(m_{\xi}^\tau(0)<\infty\big)
\ \geq\ 
\nu_\eta^\text{inv}\otimes\mathcal{P}_{\lambda,P}
\big(m_{\xi}^\tau(0)=0\big)
\\
&\ =\ 
\frac 1 n\sum_{y\in B_0}
\nu_{\eta}\otimes\mathcal{P}_{\lambda,P}
\big(m_{\xi}^\tau(y)=0\big)
\ \geq\ 
\nu_{\eta}\otimes\mathcal{P}_{\lambda,P}
\big(\forall x\in B_0\,,\ m_\xi^\tau(x)=0\big)
\,,
\end{align*}
which is strictly positive by Corollary~\ref{cor}.
Thus, the model  with initial distribution~$\nu_\eta^\text{inv}$ fixates
with positive probability.

Since~$\nu_\eta^\text{inv}$ is translation-ergodic,
Theorem~\ref{thm_phase_transition} allows us to deduce that this
initial distribution is not supercritical, that is to
say,~$\nu_\eta^\text{inv}[\xi(0)\big]\leq\zeta_c$.
Given~\eqref{density_eta0} and recalling that~$\|\eta\|\geq\zeta n$,
we obtain that~\smash{$q\zeta\leq\zeta_c$}.
This being true for every~$q<p$, we eventually deduce
that~$p\zeta\leq\zeta_c$, which concludes the proof of
Theorem~\ref{thm_no_exit}.
\end{proof}

\section{Proof of Theorem~\ref{thm_explicit}: a fraction jumps out of the
segment}
\label{section_explicit}

The aim of this section is to deduce Theorem~\ref{thm_explicit} from
Theorem~\ref{thm_no_exit}.

\subsection{Preliminary: a no man's
land around a segment}

The following Lemma tells us that adding empty intervals around the
segment~$V_n$ where particles are not allowed to sleep and
stabilizing the configuration in~$V_n$ and in these intervals does not
increase the number of particles which exit during stabilization, at
least in distribution.

\begin{lemma}
\label{lemma_NML}
Let~$\lambda>0$ and let~$P$ be a nearest-neighbour jump distribution
on~$\Z$.
Let~$n\geq 1$, let~$\eta:V_n\to\N$ be a fixed deterministic
initial configuration on~$V_n$ with only active particles, and
let~$a,\,b\in\Z$ be such
that~\smash{$W=\{a,\,\ldots,\,b\}\supset V_n$}.
Starting from the initial configuration~$\eta$ and performing legal
topplings in~$V_n$ and acceptable topplings
in~$W\setminus V_n$ until the resulting configuration is stable
in~$V_n$ and empty in~$W\setminus V_n$, we denote by~$M_n^W$ the
number of particles which jump out of~$W$.
Then~$M_n$ stochastically dominates~$M_n^W$.
\end{lemma}

This result can seem counter-intuitive, since forcing
particles to not only exit~$V_n$ but also to cross these empty intervals
could lead some of them to come back inside~$V_n$ and wake up sleeping
particles, causing more of them to jump out.
Indeed, for a fixed realization of the array of toppling instructions,
it is not true in general that~$M_n^W\leq M_n$, but the lemma
indicates that this inequality turns out to hold in distribution.

\begin{proof}
Let~$\lambda,\,P,\,n,\,\eta,\,a,\,b,\,W$ be as in the statement.
Let us first consider the simpler case of a no man's land being added
only on one side of~$V_n$, say the left side.
That is to say, we assume that~$b=\max V_n$.

Let~$\tau$ be a random array of independent instructions, with no
sleep instructions in~$W\setminus V_n$, that is to say,~$\tau$ is
obtained from the usual array (as considered in the site-wise
construction presented in Section~\ref{site_wise}) by removing the
sleep instructions on the sites of~$W\setminus V_n$.
Then~$M_n^W$ is the number of particles which jump out of~$W$ during
the stabilization of~$W$ with topplings which are legal for~$\tau$.

We then consider the following toppling strategy, which is represented
in Figure~\ref{fig_W}:
\begin{itemize}
\item \textbf{Step 1:} Topple the leftmost active particle in~$V_n$.
\item \textbf{Step 2:} If a particle just left~$V_n$ by the left exit,
force it to
walk with acceptable topplings until it leaves~$W$.
\item Repeat these steps 1 and 2 as long as there
remains active
particles in~$V_n$.
\end{itemize}

\begin{figure}
\begin{center}
\begin{tikzpicture}[scale=0.2]
\draw[very thick,->,yellow!40!black] (-27,18) -- node[midway,left]{
\shortstack{
if a particle\\
exits~$V_n$\\
by the left
}
} (-27,8);
\draw[very thick,->,yellow!40!black] (-19,18) to[out=-90,in=180] (-17.5,14)
node[below]{
\shortstack{
otherwise if~$V_n$\\
is still unstable
}
} to[out=0,in=-90] (-16,17);
\draw[very thick,->,yellow!40!black] (5,18) to[bend right=20]
node[midway,above right]{otherwise} (15,13.5);
\draw (15,12.5) node[right] 
{End of the procedure};
\draw[very thick,->,magenta!50!black] (-7.5,7) -- node[midway,right]{
\shortstack{
if~$V_n$ is\\
still unstable
}} (-7.5,17);
\draw[very thick,->,magenta!50!black] (5,7) to[bend left=20] node[midway,below right]{otherwise}
(15,11.5);
\filldraw[thick,fill=gray!60!yellow!15,draw=yellow!40!black] (-30,18) rectangle (11,28);
\filldraw[thick,fill=gray!60!magenta!15,draw=magenta!50!black] (-30,-7) rectangle (11,7);
\draw[thick, dashed,<->,gray!70!black] (-9.5,-1) --
node[below]{$V_n$} (9.5,-1);
\draw[thick, dashed,<->,gray!70!black] (-9.5,21) --
node[below]{$V_n$} (9.5,21);
\draw[thick, dashed,<->,gray!70!black] (-28.5,-4) --
node[below]{$W$} (9.5,-4);
\draw[very thick] (-9.4,.3) -- (9.4,.3);
\draw[very thick] (-9.4,22.3) -- (9.4,22.3);
\draw[very thick,dotted,blue!50, line cap=round] (-9.6,.3)
--node[below]{no sleep instructions} (-28.4,.3);
\draw[very thick,dotted,blue!50, line cap=round] (-9.6,22.3)
-- (-28.4,22.3);
\foreach \x in
{(2,1),(5,1),(6,1),(6,2),(8,1),(-6,1),(-3,1),(-2,1),(6,23),(6,24),(8,23),(5,23),(3,23),(5,1)}
{
\fill \x circle(.4);
}
\fill[yellow!40!black] (1,23) circle(.4);
\draw[very thick,->,yellow!40!black] (1,25) -- (1,23.5);
\foreach \x in {-6,-5,-2}
{
\draw (\x,23) node{$\s$};
}
\fill[magenta!50!black] (-10,1) circle(.4);
\draw[very thick,->,magenta!50!black] (-10.4,1) --  (-18,1)
to[out=180,in=180] (-18,2)
--  (1,2) to[out=0,in=0] (1,3) -- (-28.5,3);
\draw[magenta!50!black] (-9.5,5) node{\textbf{Step 2:} force the particle to walk out
of~$W$};
\draw[yellow!40!black] (-9.5,26) node{\textbf{Step 1:} topple the leftmost active
particle in~$V_n$};
\end{tikzpicture}
\end{center}
\caption{
\label{fig_W}
Strategy to prove~Lemma~\ref{lemma_NML} in the case~$\max W=\max V_n$.
The key point is that we always topple the leftmost active particle
in~$V_n$, so that whenever a particle jumps out of~$V_n$ from the left,
all the other particles are active, allowing us to force this particle
to walk out of~$W$ with no effect on the other particles.
}
\end{figure}
Let~$N$ be the number of particles which jump out of~$W$ during this
procedure.
If the procedure terminates (which happens almost surely), it yields an acceptable sequence of topplings~$\alpha$ which
stabilizes~$W$.
Thus, by the monotonicity property given by
Lemma~\ref{lemma_monotonicity}, the odometer of this sequence
satisfies~$m_\alpha\geq m_{W,\,\eta}^\tau$.
Yet, for any given realization of the array of instructions, the
number of particles which jump out of~$W$ when applying a toppling
procedure is an increasing function of the odometer of this procedure.
Therefore, we have~$N\geq M_n^W$.

We now show that~$N$ has the same distribution as~$M_n$.

Note that, after each time that step~1
is performed, then all the sleeping particles are located on the left
of the leftmost active particle.
That is to say, for any two sites~$x,\,y\in V_n$ with~$x<y$, it cannot be
that~$x$ contains an active particle while~$y$ contains a sleeping
one.
Indeed, if this was not true, then consider the first instant
when two such sites~$x<y$ exist.
Then necessarily the last instruction used must have been a sleep
instruction at~$y$, which contradicts the rule that we always topple
the leftmost active site, since~$x$ was active.

As a consequence, each time that step~2 is
triggered then all the particles in~$V_n$ must be active.
Indeed, when a particle leaves~$V_n$ by the left exit during step~1,
we know that we just toppled the leftmost site of~$V_n$,
since we consider a nearest-neighbour jump distribution.
Following the above observation, this implies that all the particles
in~$V_n$ are active.

Thus, the configuration inside~$V_n$ is left unchanged after
performing step~2, the only change in~$V_n$ being that step~2 has used some
toppling instructions, but the remaining instructions remain i.i.d.\
with the same distribution.
Let~$\tau'$ be the field of instructions
obtained from~$\tau$ by removing the instructions used during all the
occurrences of step~2.
Denote by~$M_n(\tau')$ the number of particles which
jump out of~$V_n$ during the stabilization of~$V_n$ using the
instructions in~$\tau'$, ignoring particles once they jump out
of~$V_n$.
Then, we have~$N=M_n(\tau')$.
Yet, this field~$\tau'$ has the same distribution as~$\tau$,
whence the equality in
distribution~\smash{$M_n(\tau')\stackrel{d}{=}M_n(\tau)=M_n$}.
To sum up, we have~\smash{$M_n^W\leq
N=M_n(\tau')\stackrel{d}{=}M_n$}, which shows that~$M_n$
stochastically dominates~$M_n^W$, concluding the proof in the
case~$b=\max V_n$.

We now turn to the general case of~$W=\{a,\,\ldots,\,b\}\supset V_n$.
Writing~$U=\{a,\,\ldots,\,\max V_n\}$, the above proof shows
that~$M_n$ stochastically dominates~$M_n^U$.
Then, we repeat a similar strategy to show that~$M_n^U$
dominates~$M_n^W$.
More precisely, considering an array~$\tau$ with no sleeps out
of~$V_n$, we now adopt the following strategy:
\begin{itemize}
\item \textbf{Step 1:} Topple the rightmost active particle in~$U$.
\item \textbf{Step 2:} If a particle just left~$U$ by the right exit,
force it to walk with acceptable topplings until it leaves~$W$.
\item Repeat these steps 1 and 2 as long as there
remains active particles in~$U$.
\end{itemize}
Then, the same arguments as before show that the number of particles
which exit~$W$ with this procedure is at least~$M_n^W$ and is
distributed as~$M_n^U$, which concludes the proof.
\end{proof}

\subsection{If few particles jump out, then no one leaves a slightly larger
segment}

We now prove the following lower bound on the
cost to stabilize in the no man's lands all the particles which jump out
of~$V_n$:

\begin{lemma}
\label{lemma_spread}
In dimension~$d=1$, for every~$\lambda>0$ and every nearest-neighbour jump
distribution~$P$, for every~$n\geq 1$ and every deterministic initial
configuration~$\eta:V_n\to\N$, for any~$k,\,\ell\in\N$, we have
$$
\PP_{\eta}\big(M_{n+4\ell}=0\big)
\ \geq\ 
\PP_\eta\big(M_n\leq k\big)
\times
\PP\big(G_1+\cdots+G_k\leq\ell\big)
\,,
$$
where~$(G_j)_{j\geq 1}$ are i.i.d.\
Geometric variables with parameter~$\lambda/(1+\lambda)$.
 \end{lemma}

Note that, when we write~$\PP_{\eta}\big(M_{n+4\ell}=0\big)$, we
implicitly extend the configuration~$\eta:V_n\to\N$ to the configuration
on~$V_{n+4\ell}$ which coincides with~$\eta$ on~$V_n$ and has no
particles on~$V_{n+4\ell}\setminus V_n$.

\begin{proof}
Let~$\lambda,\,P,\,n,\,\eta,\,k,\,\ell$ and~$(G_j)_{j\geq 1}$ be
as in the statement.
To stabilize~$\eta$ in~$V_{n+4\ell}$, we proceed in two steps, as
explained in the sketch of the proof in Section~\ref{section_sketches}.

First, we perform legal topplings in~$V_n$ and acceptable
topplings in~$V_{n+2\ell}\setminus V_n$ until all the sites of~$V_n$
are stable
and all the sites of~$V_{n+2\ell}\setminus V_n$ are empty.
Let us denote by~$M'_n$ the number of particles which jump out
of~$V_{n+2\ell}$ during this step.
It follows from Lemma~\ref{lemma_NML} that~$M_n$ stochastically
dominates~$M'_n$ (note that~$M'_n$
corresponds to~\smash{$M_n^{V_{m+2\ell}}$} in the notation of
Lemma~\ref{lemma_NML}).
Thus, we have
\begin{equation}
\label{dom_stoch_Mn}
\PP_\eta\big(M_n'\leq
k\big)
\ \geq\ 
\PP_\eta\big(M_n\leq k\big)
\,.
\end{equation}

Then, in the second stage, we try to stabilize these~$M'_n$ particles
inside~$V_{n+4\ell}\setminus V_n$, using the trapping procedure
introduced in~\cite{RS12} (in the proof of their Theorem~2).
This procedure shows that, if~$M'_n\leq k$, then with probability at
least~$\PP\big(G_1+\cdots+G_k\leq\ell\big)$, this second stage
succeeds, yielding an acceptable toppling sequence which stabilizes
these~$M'_n$ particles inside~$V_{n+4\ell}\setminus V_n$ with none of these
particles jumping out of~$V_{n+4\ell}\setminus V_n$.

Thus, we deduce that
$$
\PP_{\eta}\big(M_{n+4\ell}=0\big)
\ \geq\ 
\PP_\eta\big(M_n'\leq k\big)
\times
\PP\big(G_1+\cdots+G_k\leq\ell\big)
$$
which, combined with~\eqref{dom_stoch_Mn}, concludes the proof of the
Lemma.
\end{proof}

\subsection{Concluding proof of Theorem~\ref{thm_explicit}}

We now put the pieces together to obtain the claimed bound.

\begin{proof}[Proof of Theorem~\ref{thm_explicit}]
Let~$\lambda>0$, let~$P$ be a nearest-neighbour jump distribution
on~$\Z$,
let~$\zeta>\zeta_c$ and consider~$\varepsilon$,~$\alpha$ and~$\beta$
such that
\begin{equation}
\label{condition_alpha_beta}
0
\ \leq\ 
\frac{(1+\lambda)\varepsilon}{\lambda}
\ <\ 
\alpha
\ <\ 
\beta
\ <\ 
\frac{\zeta-\zeta_c}{4\zeta_c}
\,.
\end{equation}
Let~$(G_j)_{j\geq 1}$ be i.i.d.\
Geometric variables with parameter~$\lambda/(1+\lambda)$.
Since~$(1+\lambda)\varepsilon/\lambda<\alpha$, 
the weak law of large numbers ensures that
$$
\lim\limits_{n\to\infty}\,
\PP\big(
G_1+\cdots+G_{\lfloor\varepsilon
n\rfloor}
\,\leq\,
\alpha n
\big)
\ =\ 
1
\,.
$$
Thus, we can take~$n_0\geq 1$ such that, for every~$n\geq
n_0$,
$$
\PP\big(
G_1+\cdots+G_{\lfloor\varepsilon
n\rfloor}\,\leq\,\alpha n
\big)
\ \geq\ 
\frac{1+4\alpha}{1+4\beta}
\,.
$$
Now, let~$n\geq n_0$ and let~$\eta:V_n\to\N$ be a fixed
deterministic initial configuration such that~$\|\eta\|\geq\zeta
n$.
Applying Lemma~\ref{lemma_spread} with~$k=\lfloor\varepsilon n\rfloor$
and~$\ell=\lfloor\alpha n\rfloor$, we get
\begin{equation}
\label{use_lemma_spread}
\PP_\eta\big(M_{n+4\ell}=0\big)
\ \geq\ 
\PP_\eta\big(M_n\leq\varepsilon n\big)
\times
\PP\big(
G_1+\cdots+G_{\lfloor\varepsilon
n\rfloor}\,\leq\,\alpha n
\big)
\ \geq\ 
\PP_\eta\big(M_n\leq\varepsilon n\big)
\times
\frac{1+4\alpha}{1+4\beta}
\,.
\end{equation}
Then, note that
$$
\frac{\|\eta\|}{n+4\ell}
\ \geq\ 
\frac{\zeta}{1+4\alpha}
\ >\ 
\zeta_c
\,,
$$
by
virtue of~\eqref{condition_alpha_beta}.
Thus, applying Theorem~\ref{thm_no_exit} to~$\eta$, seen as a
configuration on~$V_{n+4\ell}$, we have
\begin{equation}
\label{use_thm_no_exit}
\PP_\eta\big(M_{n+4\ell}=0\big)
\ \leq\ 
\frac{(1+4\alpha)\zeta_c}{\zeta}
\,.
\end{equation}
Combining~\eqref{use_lemma_spread} and~\eqref{use_thm_no_exit}, we get
$$
\PP_\eta\big(M_n\leq\varepsilon n\big)
\ \leq\ 
\frac{1+4\beta}{1+4\alpha}
\times
\frac{(1+4\alpha)\zeta_c}{\zeta}
\ =\ 
\frac{(1+4\beta)\zeta_c}{\zeta}
\,.
$$
Taking the supremum over all
configurations~$\eta:V_n\to\N$ with~$\|\eta\|\geq\zeta n$, we
obtain that
$$
\forall\beta\in
\bigg(\frac{(1+\lambda)\varepsilon}{\lambda},\,\frac{\zeta-\zeta_c}{4\zeta_c}\bigg)\quad
\exists\, n_0\geq 1\quad
\forall n\geq n_0\qquad
\sup_{\substack{\eta:V_n\to\N\,:\\ \|\eta\|\geq\zeta n}}\,
\PP_{\eta}\big(M_n\leq \varepsilon n\big)
\ \leq\ 
\frac{(1+4\beta)\zeta_c}{\zeta}
\,,
$$
which is precisely the claim of Theorem~\ref{thm_explicit}.
\end{proof}

\section{Proof of Theorem~\ref{main_thm}}
\label{section_main}

We now prove the equivalence claimed in Theorem~\ref{main_thm}.
Let~$\lambda>0$, let~$P$ be a nearest-neighbour translation-invariant
jump distribution
on~$\Z$ and let~$\eta_0$ be an i.i.d.\ initial distribution with
mean~$\zeta$ and all particles initially active.

\subsection{Direct implication}

The direct implication is an easy consequence of
Theorem~\ref{thm_explicit}.
It follows from the Central Limit Theorem
that~$\PP\big(\|\eta_0\|_{V_n}\geq\zeta n\big)\to 1/2$
when~$n\to\infty$.
Thus, choosing whatever~$\varepsilon>0$ in the range indicated by
Theorem~\ref{thm_explicit}, we have
$$
\liminf_{n\to\infty}\,
\frac{\E\big[M_n\big]}{n}
\ \geq\ 
\liminf_{n\to\infty}\,
\PP\big(\|\eta_0\|_{V_n}\geq\zeta n\big)
\,
\inf_{\substack{\eta:V_n\to\N\,:\\ \|\eta\|\geq\zeta n}}\,
\PP_\eta\big(M_n>\varepsilon n\big)\,
\varepsilon
\ \geq\ 
\frac 1 2
\times
\Bigg[1-
\frac{\zeta_c}{\zeta}
\bigg(1+\frac{4(1+\lambda)\varepsilon}{\lambda}\bigg)
\Bigg]
\times\varepsilon
\ >\ 
0
\,.
$$

\subsection{Reciprocal: proof of Proposition~\ref{fact_critical}}

The reciprocal implication of Theorem~\ref{main_thm} follows from
Theorem~\ref{thm_condition_Mn} and Proposition~\ref{fact_critical},
which deals with the particular case of~$\zeta=\zeta_c$, and which we
now prove.

\begin{proof}[Proof of Proposition~\ref{fact_critical}]
Let~$d,\,\lambda,\,P$ and~$\eta_0$ be as in the statement.
By monotonicity (see for example Lemma~2.5 of~\cite{Rolla20}), we can
assume without loss of generality that all the particles are active in
the configuration~$\eta_0$.
Let~$\varepsilon\in(0,\,\zeta_c)$.
Then, let us consider another i.i.d.\ initial
distribution~$\eta'_0$ with mean~$\zeta_c-\varepsilon$, which
is coupled with~$\eta_0$ in such a way
that~$\eta'_0(x)\leq\eta_0(x)$ for every~$x\in\Z^d$.
This can be done for example by
taking~$\eta'_0(x)=Y_x\eta_0(x)$, where~$(Y_x)_{x\in\Z^d}$ are
i.i.d.\ Bernoulli variables with
parameter~$(\zeta_c-\varepsilon)/\zeta_c$, independent of everything
else.

Then, for every~$n\geq 1$, denoting by~$M_n$
and~$M'_n$
the numbers of particles which jump out of the box~$V_n$ starting
respectively with~$\eta_0$ and with~$\eta'_0$, we claim
that
\begin{equation}
\label{comparison_Mn}
\E\big[M_n\big]
\ \leq\ 
\varepsilon\, |V_n|
+\E\big[M'_n\big]
\,.
\end{equation}
Indeed, starting from the configuration~$\eta_0$, we may first apply
acceptable topplings to the configuration~$\eta_0-\eta'_0$, until
all particles exit, leaving us with only the
configuration~$\eta'_0$ remaining inside~$V_n$. During this
first stage, the average number of particles which jump out of the box
is equal
to~\smash{$\E\big[\|\eta_0\|-\|\eta'_0\|\big]=\varepsilon\,|V_n|$}.
Then, we stabilize in~$V_n$ with legal topplings, which gives a number
of particles jumping out of the box which is distributed
as~\smash{$M'_n$}.
Since we performed acceptable topplings during the first stage, we
obtain an upper bound on~$M_n$,
whence~\eqref{comparison_Mn}.

Then, by the contrapositive of
Theorem~\ref{thm_condition_Mn}, combined with
Theorem~\ref{thm_phase_transition}, we know that
$$
\lim_{n\to\infty}\,
\frac{\E\big[M'_n\big]}{|V_n|}
\ =\ 
0
\,.
$$
Combining this with~\eqref{comparison_Mn}, we deduce that
$$
\limsup_{n\to\infty}\,
\frac{\E\big[M_n\big]}{|V_n|}
\ \leq\ 
\varepsilon
\,+\,
\lim_{n\to\infty}\,
\frac{\E\big[M'_n\big]}{|V_n|}
\ =\ 
\varepsilon
\,.
$$
This being true for every~$\varepsilon\in(0,\,\zeta_c)$,
the proof of Proposition~\ref{fact_critical} is complete.
\end{proof}

\subsection*{Acknowledgements}

We wish to thank the anonymous referee for his valuable comments which
helped to improve the presentation.

\bibliography{article.bib}
\end{document}